\documentclass[12pt]{amsart}
\usepackage{amsmath}
\usepackage{amsthm}
\usepackage{amsfonts}
\usepackage{amssymb}
\newtheorem{Theorem}{Theorem}[section]
\newtheorem{Proposition}[Theorem]{Proposition}
\newtheorem{Lemma}[Theorem]{Lemma}
\newtheorem{Corollary}[Theorem]{Corollary}
\theoremstyle{definition}
\newtheorem{Definition}{Definition}[section]
\theoremstyle{remark}

\numberwithin{equation}{section}

\newcommand{\Z}{{\mathbb Z}}
\newcommand{\R}{{\mathbb R}}
\newcommand{\C}{{\mathbb C}}

\newcommand{\SL}{{\mathcal {SL}}}

\begin{document}

\title[Toda maps]{Toda maps, cocycles, and canonical systems}

\author{Christian Remling}

\address{Mathematics Department\\
University of Oklahoma\\
Norman, OK 73019}

\email{christian.remling@ou.edu}

\urladdr{www.math.ou.edu/$\sim$cremling}

\date{November 30, 2017; revised January 12, 2018}



\begin{abstract}
I present a discussion of the hierarchy of Toda flows that gives center stage
to the associated cocycles and the maps they induce on the $m$ functions.
In the second part, these ideas are then applied to canonical systems;
an important feature of this discussion will be my proposal that the role of the shift on Jacobi matrices
should now be taken over by the more general class of twisted shifts.
\end{abstract}
\maketitle
\section{Introduction}
The first part of this paper will present and advertise a certain new view of Toda flows.
In the second part, I then discuss flows on canonical systems from this point
of view, or rather make some suggestions; this is a completely new topic that does not seem
to have received much attention yet.

The basic idea will be to take the shift map on Jacobi matrices and its transfer matrix cocycle as the starting point and then extend
this to a cocycle for the action of a larger group (see the next section please for the precise definitions). This structure displays
a remarkable amount of rigidity in great generality: any such joint cocycle will update the Titchmarsh-Weyl $m$ functions along the action (Theorem \ref{T1.4}),
and this in turn will imply that the group action has to preserve the spectral properties and the reflection coefficients of the Jacobi
matrices it acts on (Theorem \ref{T1.2}). Furthermore, only finite gap type operators can have periodic orbits (Theorem \ref{T1.3}).
Moreover, with small modifications,
all this works for the action of any group that contains an infinite cyclic group (this is the part that will be acting by shifts)
as a normal subgroup; in this paper and for the discussion of the classical hierarchies, though, we will only need $G=\R^N\times\Z$,
with $1\le N\le\infty$.

Toda flows and more general integrable hierarchies
are a classical subject that has been studied extensively; see, for example, \cite{Dickey,GeHol1,GeHol2,Teschl} for textbook
style treatments and \cite{Dametal,DamGold} for recent work and further references. Since we will
develop it more or less from scratch here, it is clear that not everything in this paper can
be completely new, and indeed, there will be considerable overlap with what has been
done in the literature. However, since I approach the whole subject from a point of view that is not the
usual one, I hope that even those parts that discuss very well known facts
(such as the statement that any two Toda flows commute, or that they act by unitary conjugation)
will be of some interest.

The following section continues this introduction: it gives a non-technical overview of the general ideas; further details as well as the
proofs will be discussed in Sections 3--5. Finally, in Section 6, I use these ideas as a guideline to attempt
a similar approach to evolutions of \textit{general }canonical systems, which is a new topic that (strange as that may sound)
seems to have received hardly any attention yet. I have emphasized the word \textit{general }here: if only a subclass of special
canonical systems is evolved, then one is back on familiar territory and in fact the Toda, KdV, AKNS hierarchies are all of this type.
My main point in Section 6 will probably be the proposal to replace
the shift by a more general family of flows and make the flows to be constructed commute with these
rather than the shift itself.

\textit{Acknowledgment: }I'd like to thank Injo Hur and Darren Ong for inspiring discussions on these topics
and the anonymous referee for a very careful reading of the paper and helpful comments.
\section{Toda flows, cocycles, and Toda maps}
A \textit{Jacobi matrix }is a difference operator of the form
\begin{equation}
\label{jac}
(Ju)_n = a_n u_{n+1} + a_{n-1}u_{n-1} + b_n u_n .
\end{equation}
Here, we assume that $a_n > 0$ and $b_n\in\mathbb R$ are bounded sequences, and then
$J$ is a bounded self-adjoint operator on $\ell^2(\Z)$.

It is convenient to have the separate notation $\tau u$ available when
the difference expression from \eqref{jac} is applied
to an arbitrary sequence $u$, not necessarily in $\ell^2$.

We denote by $\mathcal J$ the space of all such (bounded) Jacobi matrices, and we endow it with the metric
\begin{equation}
\label{djac}
d(J,J') = \sum_{n\in\Z} 2^{-|n|}\left( |a_n-a'_n|+|b_n-b'_n|\right) .
\end{equation}
This metric is often much more useful than the operator norm because it tends to
make sets compact more easily, and it interacts well with other quantities of interest in spectral theory such as
spectral measures and $m$ functions.

\textit{Toda flows }are global flows on $\mathcal J$.
Alternatively, and this is my preferred point of view, we can think of the Toda hierarchy as an action of the (abelian) group
of polynomials $\mathcal P=\mathbb R[x]$ (with pointwise addition as the group operation) on the space $\mathcal J$ of Jacobi matrices.
Before we discuss the approach via cocycles
in detail, let me quickly review what is probably the most common way to construct the Toda hierarchy. Suppose
that $p\in\mathcal P$ is given. Then the \textit{Lax equation}
\begin{equation}
\label{toda}
\dot{J} = [p(J)_a, J]
\end{equation}
defines a global flow \cite[Theorem 12.6]{Teschl}, and we can then define $p\cdot J$ as the time one map of this flow. Here, the anti-symmetric part
$p(J)_a$ of $p(J)$ is defined via its matrix representation in terms of the standard unit vectors $\delta_n\in\ell^2$:
if we write $X_{jk}=\langle \delta_j, X\delta_k\rangle$ for a bounded self-adjoint operator $X$,
then $(X_a)_{jk}=X_{jk}$ for $j<k$ and $=-X_{jk}$ if $j>k$, and $(X_a)_{jj}=0$.

One can then establish (with a fair amount of effort, that is) the following properties:
First of all, these flows all commute with each other, and this we already anticipated by saying that the abelian group
$\mathcal P$ acts on $\mathcal J$. (This commutativity does then ensure that we have a group action because the
right-hand side of \eqref{toda} is linear in $p$.)
Moreover, they also commute with the shift $S$ that sends a Jacobi matrix $J$ to $SJ$, which has
shifted coefficients $(a_{n+1},b_{n+1})$. So we have actually obtained an action of the larger abelian group $G=\mathcal P\times\Z$,
with the second factor acting by shifts.
Finally, $g\cdot J = U^*JU$ for some unitary $U=U(g,J)$. (All these properties will be given alternative explanations
later in this section and the next as we develop the material.)

The maps $J\mapsto g\cdot J$ are continuous with respect to $d$, in the sense spelled out below;
this property will be important for us, and it is usually not mentioned in the standard treatments,
so I'll state it separately and also provide a proof later, in Section 3.
\begin{Theorem}
\label{T2.1}
For any $R>0$ and $p\in\mathcal P$, the map $J\mapsto p\cdot J$ is a homeomorphism of $(\mathcal J_R, d)$,
where $\mathcal J_R = \{ J\in\mathcal J: \|J\|\le R \}$.
\end{Theorem}
What is usually discussed is the continuity with respect
to the operator norm, and a similar argument will establish Theorem \ref{T2.1}.
Observe also that since $p\cdot J$ is unitarily equivalent to $J$, the operator norm is preserved and $\mathcal J_R$ is indeed
invariant under the action.

The central objects of this paper are cocycles (in the sense the word is used in dynamical systems). Let me review the definition.
A cocycle takes values in a group, and for us, a special group of matrix functions will be of great importance, so let's introduce the name
\[
\SL = \{ T:\C \to SL(2,\C) : T(z) \:\textrm{\rm entire, } T(x)\in SL(2,\R) \textrm{ for }x\in\R \}
\]
for this; also, note that $\SL$ indeed
is a group with the obvious operation of pointwise matrix multiplication.

\begin{Definition}
An \textit{$\SL$-cocycle }associated with an action of a group $G$ on a space $X$ is a map $T: G\times X\to\SL$
satisfying the \textit{cocycle identity}
\begin{equation}
\label{ccid}
T(gh, x) = T(g, h\cdot x) T(h, x) .
\end{equation}
\end{Definition}
We will occasionally encounter cocycles taking values in other groups such as $\textrm{SL}(2,\C)$, $\textrm{SL}(2,\R)$, $\C^{\times}$,
and of course these are defined in the same way. The reader interested in cocycles in general may consult \cite{Feres,Katok,KHas} for more
information, but I should also issue a warning that while this material makes for interesting reading, it will be largely irrelevant to what
we do here.

For our purposes, the most fundamental example of an $\SL$-cocycle is given by the transfer matrix $T(n;J)$. In this case,
$G=\Z$, and this acts by shifts on $X=\mathcal J$, that is, $n\cdot J= S^n J$. Now $T$ is defined as
\begin{equation}
\label{shiftcc}
T(1;J) = A(J), \quad
A(J)\equiv\begin{pmatrix} \frac{z-b_1}{a_1} & \frac{1}{a_1} \\ -a_1 & 0 \end{pmatrix} ,
\end{equation}
and then $T(n;J)=A((n-1)\cdot J)A((n-2)\cdot J) \cdots A(J)$
and $T(-n;J)=T(n;(-n)\cdot J)^{-1}$ for general $n\ge 1$, and $T(0;J)=1$.
These latter definitions, having chosen $T(1;J)$, are of course forced on us by the cocycle identity.

It is then easily verified that $T$ is an $\SL$-cocycle, and it is also clear that,
conversely, any cocycle for an action of $G=\Z$ will be of this form,
for some matrix function $A:\mathcal J\to\SL$, which we can recover from the cocycle as $A(J)=T(1;J)$.

Occasionally, we will consider $T$ for a fixed $z\in\C$, and then it becomes an $\textrm{SL}(2,\C)$- or, if $z\in\R$,
$\textrm{SL}(2,\R)$-cocycle.

Of course, this is a rather highbrow view of the transfer matrix $T$;
often it is perfectly appropriate to just think of $T$ as the matrix that updates solution vectors, as follows:
if $u$ solves the difference equation $\tau u=zu$ and $Y(n)=(u_{n+1},-a_nu_n)^t$, then
$T(n)Y(0)=Y(n)$.

The \textit{Titchmarsh-Weyl $m$-functions }$m_{\pm}$, one for each half line $\Z_{\pm}$, are defined for
$z\in\C^+=\{z\in\C : \textrm{Im}\: z>0 \}$ by
\[
m_{\pm}(z) = \mp \frac{f_{\pm}(1,z)}{a_0 f_{\pm}(0,z)} ,
\]
where $f_{\pm}$ solve $\tau f=zf$ and $f_{\pm}\in\ell^2(\Z_{\pm})$. Notice that if we identify a vector $v\in\C^2$, $v\not=0$,
with the point $v_1/v_2$ on the Riemann sphere $\C_{\infty}=\C\cup \{\infty\}$,
then we can say that $m_{\pm}(z) = \pm F_{\pm}(0,z)$, where, as above,
$F(n)=(f_{n+1},-a_nf_n)^t$. In fact, this is the reason we defined $F$ and $T$ in this way.

This has the important consequence that $T$ is not just a cocycle of arbitrarily evolving matrices; rather, the cocycle
updates $m_{\pm}$ correctly along the action of $\Z$ in the sense that for $z\in\C^+$,
\begin{equation}
\label{2.71}
\pm m_{\pm}(z; n\cdot J) = T(n;J)(\pm m_{\pm}(z; J)) .
\end{equation}
Here, a matrix $A=\left( \begin{smallmatrix} a & b \\ c & d \end{smallmatrix} \right)$ acts on the Riemann sphere as a linear fractional
transformation: $Aw=\frac{aw+b}{cw+d}$.
Throughout this paper, an expression of this type (that is, a matrix followed by a number) will always refer to this action.
In other contexts, group actions will usually be indicated by the dot notation that we have already started employing above.

One more trivial but important point is worth making here: namely, the action of an $A\in\textrm{SL}(2,\C)$ on $\C_{\infty}$ is
consistent with its action on vectors $v\in\C^2$ in the sense that if we apply $A$ to such a $v\not=0$ to obtain $Av\in\C^2$, but then
change our mind and would rather let $A$ act as a linear fractional transformation on the point from $\C_{\infty}$ that $v$ represents,
then it suffices to also reinterpret $Av\in\C^2$ as a point from $\C_{\infty}$.

The Toda flows similarly come with associated cocycles that have the same basic properties.
Let me report quickly on this; the construction is discussed in detail in \cite[Section 12.4]{Teschl}.
Note, however, that this reference uses slightly different conventions (on what $T$ is, for example) and a completely
different perspective; in particular, cocycles are not mentioned in \cite{Teschl}.

Consider a fixed flow of the Toda hierarchy; here, we really want to think of this as an action of $G=\R$
on $\mathcal J$. Then there is a matrix function $B=B(z,J)$, entire in $z$ and continuous in $J$, and real on the real line,
with $\textrm{tr}\: B=0$, such that the following holds:
If $T=T(t;J)$ is defined as the solution of
\begin{equation}
\label{todacc}
\dot{T} = B(z, t\cdot J) T, \quad T(0)=1,
\end{equation}
then $T(t;J)$ is an $\SL$-cocycle for the action of $G=\R$ on $\mathcal J$, and again
\begin{equation}
\label{2.84}
\pm m_{\pm}(z; t\cdot J) = T(t;J)(\pm m_{\pm}(z; J)) .
\end{equation}
The fact that $T$ is a cocycle follows just from the form of \eqref{todacc}, and this would in fact work
for arbitrary $B$. Conversely, a differentiable cocycle for a given flow will satisfy \eqref{todacc}, with
$B(J)=(d/dt)T(t;J)\bigr|_{t=0}$. In other words, \eqref{todacc}, for general $B$, can be thought of as a rewriting of
the cocycle property, and then we claim that an appropriate choice of $B$ will also give \eqref{2.84}.

We have explicit formulae for these matrix functions $B$, which I will review in Section 5;
see especially \eqref{todaB} below. To just give one simple concrete
example for now, consider the classical Toda flow itself ($p(x)=x$ in \eqref{toda}): then
\begin{equation}
\label{Btoda}
B(J) = \begin{pmatrix} z-b_1 & 2 \\ -2a_0^2 & b_1-z \end{pmatrix} .
\end{equation}
While we will not provide a derivation of the general formula \eqref{todaB} (see \cite[Section 12.4]{Teschl} and
\cite[Theorem A.1]{OngR} for this),
we will indicate in
Section 3 how \eqref{Btoda} can be obtained in the framework proposed in this paper.

We then have such a matrix function $B=B_p$ for any choice of the polynomial $p$, and thus we obtain a cocycle
for every flow from the hierarchy, and of course if we consider just one fixed flow, then $G=\R$ is
the group that is acting on the Jacobi matrices.
Our next result says that we in fact obtain a cocycle $T$ with respect to the action of the whole
(and much larger) group $G=\mathcal P \times \Z$; recall that $\Z$ acts by shifts here, so $(p,n)\cdot J = p\cdot S^nJ= S^n(p\cdot J)$.
Let me make the definition of $T(g;J)\in \SL$ completely explicit: if $g=(p,n)$, then $T(g;J)=T(p; n\cdot J) T(n; J)$, where
$T(n;J)$ was defined in \eqref{shiftcc} and $T(p;J)$ is constructed as described above, by solving \eqref{todacc} for $B=B_p(tp\cdot J)$
and then evaluating at $t=1$. The alternative definition $T(g;J)=T(n; p\cdot J)T(p; J)$ would have worked, too, and in fact
this is part of what the theorem says. (Of course, contrary to what I promised, this is not really a complete definition of $T$ yet; we
also need the explicit formula \eqref{todaB} for $B$ from Section 5.)
\begin{Theorem}
\label{T1.1}
The matrices $T(g;J)$ form an $\SL$-cocycle for the action of $G=\mathcal P\times\Z$.
\end{Theorem}
This is known if we just act by either the shift or an individual Toda flow or a combination of these,
but the theorem makes the much stronger
claim that the cocycle identity \eqref{ccid} also holds if $g,h$ refer to different flows from the Toda hierarchy
(possibly combined with shifts also).

The existence of a \textit{joint }cocycle in this sense, even if the acting group is just $G=\R\times\Z$,
is such a strong condition that the Toda hierarchy can be recovered
from this property, plus the additional requirement that the matrices $B$ from \eqref{todacc} are polynomials in $z$.
This will be discussed in concrete style in Section 3.

From a more philosophical point of view, I believe the next result expresses the main reason why this structure
(a joint cocycle extending the shift cocycle) displays so much rigidity. It will be easy to prove, but I find it mildly surprising nevertheless.

We will give it in a rather general form and consider an action of an abelian group $G$ on $\mathcal J$, and our only assumption
on $G$ is that it contains an element of infinite order or, equivalently, a subgroup isomorphic to $\Z$. It will be notationally convenient
to just denote this subgroup as $\Z$ itself.

This subgroup $\Z$ will act by shifts, and
the application we have in mind would be to $G=\R\times\Z$, so in addition to the shift we have a flow that commutes with it.

We could in fact consider general, possibly non-abelian groups $G$, and we would then assume that $\Z\trianglelefteq G$ is contained
in $G$ as a \textit{normal }subgroup. Essentially the same proof as the one we are going to give still works, with the only modification
that $m_{\pm}$ could also get swapped as part of the updating done by the cocycle. For example, imagine a group element acting
by a reflection of the coefficients of $J$; in fact, the corresponding group action, which is generated by the shift and a reflection, provides probably the most
basic example of such an action of a non-abelian group with a cocycle that extends the shift cocycle. In this case, the acting group
is the infinite dihedral group.

We don't have any use for this non-abelian version of Theorem \ref{T1.4} in this paper, so I won't make it
explicit, but I do think that it suggests new directions and I intend to pursue these topics in a future project.
\begin{Theorem}
\label{T1.4}
Let $G$ be an abelian group with $\Z\le G$, and consider an action of $G$ on $\mathcal J$,
with $\Z$ acting by shifts $n\cdot J=S^nJ$. Fix $z\in\C^+$, and
suppose that $T(g;J)$ is an $\textrm{\rm SL}(2,\C)$-cocycle that extends the shift cocycle (for this fixed $z$)
in the sense that $T(n;J)$ is given by \eqref{shiftcc}.
Then, if $g \in G$, $J\in\mathcal J$ satisfy
\begin{equation}
\label{1.4a}
\liminf_{n\to\pm\infty} \|T(g; n\cdot J)\| < \infty ,
\end{equation}
then $\pm m_{\pm}(z, g\cdot J) = T(g;J)(\pm m_{\pm}(z,J))$.
\end{Theorem}
The additional assumption \eqref{1.4a} is aesthetically displeasing, but it seems necessary and
it should be extremely easy to verify in all cases of interest,
and let me perhaps also make a few comments on why this should be the case:
Typically, we expect $T(g; n\cdot J)$ to be bounded, uniformly in $n\in\Z$ and also in $z$, varying over a compact subset of $\C$,
and this should follow from the orbit $\{ n\cdot J \}$ having compact closure in $\mathcal J$. Then continuity of $T(g;J)$ in $J$
would be enough to produce the desired conclusion. This argument does not literally work in complete generality
because to make $\mathcal J_R$ compact,
we would have to include Jacobi matrices with $a_n=0$ in the definition of $\mathcal J$, which is in fact often convenient and I have
done this before on other occasions, but here it has the serious drawback that then the shift cocycle is undefined on those $J$.
The Toda cocycles, on the other hand, do have continuous extensions to Jacobi matrices with $a_n=0$
(and these sets are invariant under the action of $\mathcal P$).
Also, it would not be completely ridiculous to restrict the whole treatment to Jacobi matrices satisfying a uniform bound
$a_n\ge\delta >0$ (I did this in \cite{Remac}), and then these technical problems disappear entirely.

Also, note that since we're dealing with holomorphic functions,
in the case of $\SL$-cocycles (which is really all we're interested in here),
it suffices to establish the hypotheses of Theorem \ref{T1.4} for set of $z$'s with an accumulation point in $\C^+$ to have the statement
available for all $z\in\C^+$.

Let's now discuss the statement (rather than the assumptions) of Theorem \ref{T1.4}. I believe that several remarkable things
have happened here. Originally, it is natural to think of a cocycle $T(g;x)$ as something that evolves on the side, always keeping an eye
on the base dynamics $x\mapsto g\cdot x$. Now in the situation of Theorem \ref{T1.4}, it turns out that the cocycle
itself already tells the full
story: $T$ updates $\pm m_{\pm}$, but these determine $J$, so we recover the base dynamics from $T$. The \textit{zero curvature equation }\eqref{zctoda}
will make a very explicit statement about this.

Next, the mere fact that the updating is done by a matrix function $T(z)\in\SL$ is a very strong restriction, so only very special
group actions can ever have such an associated joint cocycle that extends the shift cocycle. Let's look at this in more detail.
I defined such maps,
\[
\pm m_{\pm}(z) \mapsto T(z) (\pm m_{\pm}(z)) ,
\]
earlier in \cite{Remgenrefl}, where I called them \textit{transformations of type }$\textrm{TM}$, which, by an unlikely coincidence,
is consistent with the name \textit{Toda maps, }which I would now like to give them; my original idea in \cite{Remgenrefl} was to
let $\textrm{TM}$ stand for \textit{transfer matrix. }Examples of Toda maps are given by the matrices from a Toda or shift cocycle,
but of course there are many others.

Toda maps are most naturally considered for general Herglotz functions, not necessarily coming from a Jacobi matrix;
equivalently, by the one-to-one correspondence between Herglotz functions and (trace normed) canonical
systems \cite{dB,Pota,Win}, we can think of $T\in\SL$ as inducing a map on whole line canonical systems $Ju'=-zHu$,
$J=\left( \begin{smallmatrix} 0 & -1\\ 1 & 0 \end{smallmatrix} \right)$. I'll use these two points of view
interchangeably; in particular, a canonical system (or just its coefficient function $H(x)$) in a situation where none was introduced
before will be understood to be associated with a given pair of Herglotz functions. For now, we don't really need to understand
anything about canonical systems; we can just view $H(x)$ as a convenient short-hand notation for a pair $m_{\pm}$ of Herglotz functions.

As a final general remark,
notice that Toda maps have domains, possibly quite small (or empty) ones, since there
is no guarantee that letting $T$ act on a Herglotz function will produce another Herglotz function.
\begin{Theorem}
\label{T1.2}
Let $T\in\SL$, and suppose that $m_{\pm}^{(j)}$ are related by the Toda map
\[
\pm m_{\pm}^{(2)}(z) = T(z) (\pm m_{\pm}^{(1)}(z)) .
\]
Then $H_1$ and $H_2$ are unitarily equivalent, and the reflection coefficients satisfy $|R_1(x)|=|R_2(x)|$ for almost every $x\in\R$.
In particular, for any Borel set $A\subseteq \R$, $H_1$ will be reflectionless on $A$ if and only if $H_2$ is.
\end{Theorem}
Some clarifying comments are in order:

(1) The general reflection coefficients, which are defined for any canonical system, not necessarily of classical scattering type,
are given by
\[
R_+(z) = \frac{\overline{m_+(z)} + m_-(z)}{m_+(z)+m_-(z)} , \quad\quad
R_-(z) = \frac{m_+(z) + \overline{m_-(z)}}{m_+(z)+m_-(z)} ;
\]
since $|R_+(x)|=|R_-(x)|$ at those $x\in\R$ at which both $R_{\pm}(x)\equiv \lim_{y\to 0+} R_{\pm}(x+iy)$ exist, it doesn't matter here
which reflection coefficient we take.
Please see \cite{Remgenrefl} for more on reflection coefficients and \cite{GNP,Ryb} for earlier uses of them in somewhat
different situations.

(2) We call $H$ \textit{reflectionless }on a Borel set $A\subseteq\R$ if
$m_+(x)=-\overline{m_-(x)}$ for almost every $x\in A$. This is equivalent to $R(x)=0$ almost everywhere on $A$, and thus the final
statement of Theorem \ref{T1.2} is an immediate consequence of the earlier claim that $|R_1(x)|=|R_2(x)|$.
Please see \cite{BRSim,HMRem,PR1,Remac} for (much) more on reflectionless operators and why they are important.

(3) Recall that for a whole line operator (think of specifically a Jacobi matrix or a Schr\"odinger operator perhaps), there is
a standard way to construct a specific spectral representation from the half line $m$ functions $m_{\pm}$. Namely, define the
matrix valued Herglotz function $M$ as
\begin{equation}
\label{defM}
M(z) = \frac{1}{m_+(z) + m_-(z)} \begin{pmatrix} -1 & \frac{1}{2}(m_+(z) - m_-(z)) \\
\frac{1}{2}(m_+(z)-m_-(z)) & m_+(z)m_-(z) \end{pmatrix} ;
\end{equation}
its Herglotz representation gives us a matrix valued measure $\rho$, and the operator is then unitarily equivalent to multiplication
by the variable in $L^2(\R , d\rho)$. See, for example, \cite[Section 2.5]{Teschl} for the Jacobi case.

In our general situation, we don't have an operator (construction of these is actually more
involved for canonical systems, and we don't really need them here), so we take this as our definition of the spectral properties.
In particular, the claim of Theorem \ref{T1.2} that $H_1$, $H_2$ are unitarily equivalent really means that the operators of multiplication
by the variable in $L^2(\R , d\rho_j)$ are unitarily equivalent.

We write $\mathcal R(A)$ for the set of Jacobi matrices that are reflectionless on $A$, and then,
for a non-empty compact set $K\subseteq\R$, we define
\begin{equation}
\label{defR0}
\mathcal R_0(K) = \{ J\in\mathcal J : J\in \mathcal R(K), \sigma(J)\subseteq K \} .
\end{equation}
We can make the same definitions for canonical systems, and we will write
$\mathcal R^C(A)$ and $\mathcal R^C_0(K)$ for these larger spaces.

$\mathcal R_0(K)$ is a compact subset of $\mathcal J$ if $K$ has positive Lebesgue measure;
if $K$ is essentially closed, then $\sigma(J)=K$ for all $J\in\mathcal R_0(K)$.
If $K$ is a finite gap set (a disjoint union of finitely many closed intervals of positive finite length), then $\mathcal R_0(K)$
is the set of finite gap Jacobi matrices with spectrum $K$ and, in particular, is homeomorphic to a torus of dimension equal to
the number of (bounded) gaps of $K$.

Theorem \ref{T1.2} implies that, for arbitrary compact $K$, the space $\mathcal R_0(K)$ is invariant under maps induced
by a $T(z)\in\SL$, provided that this maps Jacobi matrices to Jacobi matrices again.
In particular, the $\mathcal R_0(K)$ are invariant under all flows from the Toda hierarchy.
The collection of these sets is dense in $\mathcal J$: in fact, the Jacobi matrices with periodic coefficients
are already dense, and such a $J$ lies in $\mathcal R_0(K)$ with $K=\sigma(J)$.
Therefore, what happens to a general initial point under a Toda flow is determined by
what happens on the $\mathcal R_0(K)$. As we will see in a moment, this leads to a very easy and transparent explanation of the fact
that any two flows from the Toda hierarchy commute with each other.

It is also natural to ask what the fixed points of a Toda map are. In the special case of a Toda flow,
it is well known that these are exactly the $J\in\mathcal R_0(K)$ for a finite gap set $K$ (where $K$ depends on the flow under consideration).
The general case has essentially the same answer:
\begin{Theorem}
\label{T1.3}
Suppose that $\pm m_{\pm}(z) = T(z)(\pm m_{\pm}(z))$ for a $T\in\SL$, $T\not\equiv \pm 1$. Then $H\in\mathcal R^C_0(E)$,
and here $E$ is a union of disjoint closed intervals (possibly unbounded or consisting of single points)
whose endpoints do not accumulate anywhere. More specifically,
\[
E=\{ x\in\R: -2\le \textrm{\rm tr }T(x) \le 2\} .
\]
\end{Theorem}
Such an $H$ automatically has empty singular continuous spectrum, the absolutely continuous spectrum is essentially
supported by $E$ and of multiplicity $2$ locally, and point spectrum is only possible at the isolated points of $E$, if any.
These statements follow from the fact that $H\in\mathcal R^C_0(E)$, if we use
general results about reflectionless Herglotz functions;
see \cite{PR1} for more on this. However, they are also very easy to verify directly here, and we will do this in Section 4,
when we prove Theorem \ref{T1.3}.

We can now give very transparent explanations of the well known basic properties of Toda flows:
By combining Theorem \ref{T1.4} with Theorem \ref{T1.2}, we see that $p\cdot J$ is unitarily equivalent to $J$
and the absolute value of the reflection coefficient is preserved. In particular, $p\cdot J\in\mathcal R(A)$ precisely if
$J\in\mathcal R(A)$ (a different and much more technical proof of this fact was earlier given in \cite{PR2}, and see
also \cite{Ryb} in this context). So the sets $\mathcal R_0(K)$ are invariant under the action of $G=\mathcal P\times\Z$.
The fixed points are finite gap
Jacobi matrices by Theorem \ref{T1.3}. Finally, by their construction (if the Toda hierarchy
is indeed constructed following the suggestions above and as outlined more explicitly in Section 3), Toda flows commute with the shift.

What is missing from this list of basic properties is the fact that any two Toda flows also commute with each other.
This would in principle follow from corresponding properties of the matrices $B=B(z,J)$ from the cocycle equations.
However, a much more intuitive and less technical explanation is also possible. This depends on the following easy general fact from
topological dynamics.
\begin{Lemma}
\label{L2.1}
Let $S$ be a homeomorphism on a compact metric space $X$, and suppose that $(X,S)$ is an equicontinuous
dynamical system with a dense orbit. Then any two continuous maps on $X$ that commute with $S$ also commute with
each other.
\end{Lemma}
The equicontinuity assumption refers to the family of maps $S^n, n\in\Z$.
\begin{proof}
Denote the two maps by $F$ and $G$, respectively.
Suppose that $\{ S^n x\}$ is dense in $X$, and let $y\in X$ be an arbitrary point. Then $y=\lim S^{k_j} x$
for a suitable sequence $k_j\in\Z$, and similarly $Fx = \lim S^{m_j}x$, $Gx=\lim S^{n_j}x$. Then
\[
Gy = G\lim S^{k_j} x = \lim GS^{k_j} x =
\lim_{j\to\infty} S^{k_j} \lim_{p\to\infty} S^{n_p} x ,
\]
and now the equicontinuity of $S^{k_j}$ implies that also $Gy = \lim S^{k_j+n_j}x$. By repeating this argument,
we find that $FGy = \lim S^{k_j+m_j+n_j}x$, and then that $GFy$ equals the same expression.
\end{proof}
\begin{Corollary}
\label{C2.1}
Any two Toda flows commute.
\end{Corollary}
\begin{proof}[Sketch of proof]
We apply the Lemma to $X=\mathcal R_0(K)$, with $K$ being a finite gap set. Then $(X,S)$ is equicontinuous; in fact,
it is well known that $(X,S)$ can be conjugated to become a translation $\tau_a$ on a finite-dimensional torus
\cite[Section 9.1]{Teschl}. The additional
assumption of Lemma \ref{L2.1}, on dense orbits, will be satisfied if $a$ has suitable properties. The claim of the Corollary
now follows from Theorem \ref{T2.1} and the fact that the collection of these $\mathcal R_0(K)$ is still dense in $\mathcal J$.
\end{proof}
\section{Reconstruction of the Toda hierarchy from cocycles}
Before we turn to the topic announced in the title of this section, let me give the proofs of Theorems \ref{T2.1}, \ref{T1.4}.
\begin{proof}[Proof of Theorem \ref{T2.1}]
For a given $p\in\mathcal P$, write $X(J)=[p(J)_a,J]$ for the right-hand side of \eqref{toda}. This is Lipschitz continuous
with respect to the operator norm and thus the standard Picard iteration technique may be employed (see \cite[Section 12.2]{Teschl}
for more details): the solution $J(t)$ of \eqref{toda} on $0\le t\le T$
with the initial value $J$ may be obtained as $J(t)=\lim J_n(t)$, with
\begin{equation}
\label{pic}
J_0(t) = J, \quad J_{n+1}(t) = J + \int_0^t X(J_n(s))\, ds ,
\end{equation}
and the convergence is in operator norm, and it is uniform in $t\in [0,T]$. The length $T$ of this interval only depends on the Lipschitz constant.

As a by-product, one also obtains the continuity of $J\mapsto p\cdot J$ with respect to the operator norm from this method, but
we would like to use $d$ instead of $\|\cdot \|$, so it's now natural to wonder if we also have a Lipschitz condition with respect to $d$,
and indeed this works: we have that
\begin{equation}
\label{lip}
d(X(J_1),X(J_2)) \le L d(J_1, J_2)
\end{equation}
for some $L$ that only depends on the polynomial $p$ and the bound $R$ on $\|J\|$. To see this, observe that
the matrix $p(J)_a$ has finite band width
in the sense that $\langle \delta_j ,p(J)_a\delta_k\rangle =0$ for $|j-k|>\deg p$. So when we compute a matrix element
$X(J)_{nk}$, then only those coefficients $a_j,b_j$ of $J$ at at most a certain distance from $n$ or $k$ are involved, and of course
$X(J)_{nk}$ is a smooth function of those coefficients, with uniformly bounded derivatives on $J\in\mathcal J_R$.
From this it follows that there
are constants $C$ and $D$, depending only on $R$ and $p$, such that
\begin{equation}
\label{1.4b}
\left| X(J)_{nk} - X(J')_{nk} \right| \le C \sum_{|j-n|\le D} \left( |a_j-a'_j| + |b_j-b'_j| \right) ;
\end{equation}
also recall that only $k=n, n\pm 1$ gives non-zero matrix elements here. By multiplying \eqref{1.4b} by $2^{-|n|}$ and summing over $n,k$,
we now indeed obtain \eqref{lip}.

The rest is routine: the argument alluded to above can simply be repeated with \eqref{lip} as the key ingredient.
To spell this out more explicitly,
consider two initial values $J,J'$, and observe that the metric is of the form $d(J_1,J_2)=\|J_1-J_2\|_w$, for
a certain weighted $\ell^1$ norm $\|\cdot \|_w$. Thus
an inductive argument using \eqref{pic} shows that
\[
d(J_n(t), J'_n(t)) \le d(J,J') \left( 1 + Lt + \ldots + \frac{(Lt)^n}{n!} \right) ,
\]
so $d(J(t),J'(t))\le e^{Lt} d(J,J')$.
\end{proof}
\begin{proof}[Proof of Theorem \ref{T1.4}]
I'll discuss the claim about $m_+$; of course, a similar argument will work for $m_-$. Recall that $M=m_+(z)$
may be characterized as the unique number that makes (the components of) $T(n;J)(M,1)^t$ square summable over $n\ge 1$.
Moreover, a solution vector $F_n =(f_{n+1},-a_nf_n)^t$, with $\tau f=zf$, will be square summable already if $\liminf \|F_n\|=0$. This follows
because the Wronskian $F_n^tJG_n$, $J=\left( \begin{smallmatrix} 0 & -1 \\ 1 & 0 \end{smallmatrix} \right)$, of two solutions
$F,G$ is constant and the $\ell^2$ solution goes to zero, so all other solutions must become large.

Consider now the solution
\[
F_n = T(n;g\cdot J)T(g;J)\begin{pmatrix} M \\ 1 \end{pmatrix} = T(g; n\cdot J) T(n;J) \begin{pmatrix} M \\ 1 \end{pmatrix} .
\]
Since, by assumption, $T(g; n\cdot J)$ stays bounded on a subsequence $n_k\to\infty$, we see that $\liminf \|F_n\|=0$.
By our preliminary remarks, this identifies $F_n$ as the $\ell^2$ solution of the Jacobi matrix $g\cdot J$.
Hence its $m$ function $m(g\cdot J)$ is the number represented by $T(g;J)(M,1)^t$, but this is
$T(g;J)M$, as claimed.
\end{proof}

We now indicate how the Toda hierarchy could be constructed starting from the
requirement that we wish our evolutions to have an associated $\SL$-cocycle when combined with the shift,
rather than from the Lax equation \eqref{toda}, as is usually done.
We do not make rigorous claims in this part and will freely use formal calculations.

We consider a single flow, unknown at this point, which we'll denote by $t\cdot J$ (in our notation from above,
this would become $tp\cdot J$, for a fixed polynomial $p$ and $t\in\R$). So the acting group now is $G=\R\times\Z$,
with $n\cdot J$ acting by shifts and $t\cdot J$ is what we're trying to construct. More careful notation for these group
elements would have been $(0,n)$ and $(t,0)$, respectively, but our abbreviated version is more pleasing to look at,
and we in fact already used similar conventions earlier.

Following our general plan,
we now make the crucial (and strong) additional
assumption that there is an associated $\SL$-cocycle $T=T(t,n;J)$ whose shift part is given by \eqref{shiftcc}.
As in \eqref{todacc}, the $t$ part is described
by a matrix function $B=B(z,J)$, $\textrm{tr}\: B=0$. We
differentiate both sides of the identity
\begin{equation}
\label{3.1}
T(t;1\cdot J)T(1;J) = T(1;t\cdot J)T(t; J)
\end{equation}
with respect to $t$ at $t=0$. Obviously, \eqref{3.1} is a special case of the cocycle identity
\eqref{ccid}; however, it is also true and easily verified that if we have individual cocycles $T_0(t;J)=T(t,0;J)$ and $S(n;J)=T(0,n;J)$,
for the actions of $\R$ and $\Z$, respectively, and the compatibility condition \eqref{3.1} holds,
then $T(t,n;J):=T_0(t; n\cdot J) S(n;J)$ defines a cocycle for the action of $\R\times\Z$.

Next, recall that $T(1;J)=A(J)$, with the $A$ from \eqref{shiftcc}, so, formally at least, we obtain that
\begin{equation}
\label{zctoda}
B(1\cdot J)A(J) = \dot{A}(J)+A(J)B(J) ,
\end{equation}
where $\dot{A}(J)$ is short-hand for $(d/dt)A(t\cdot J)\bigr|_{t=0}$. This equation is known
as the \textit{zero curvature equation. }Our derivation of it here gives it a rather transparent interpretation:
It is the cocycle property for $T(g;J)$ in differential form. More precisely, it is a compatibility
condition that will ensure that the individual cocycles form a joint cocycle when glued together.
Please see also Section 3 and especially Theorem 3.1 of \cite{OngR}, where these remarks are
made more precise.

To obtain the Toda hierarchy from \eqref{zctoda}, we make the additional assumption that $B(z,J)=\sum_{n=0}^N z^n B_n(J)$
is a polynomial in $z$. Since $A(z,J)=A_0(J)+zA_1(J)$ also is of this type, we can then compare coefficients in \eqref{zctoda}.

Let me take a quick look at the cases $N=0$ and $N=1$, as an illustration, without attempting to give a general treatment.
Such a general treatment, however, is possible and was given by Ong \cite{OngToda}.

If $N=0$, then comparing coefficients of $z^1$ in \eqref{zctoda} shows that $B_{12}=B_{21}=0$. Since $\textrm{tr}\: B=0$,
this means that $B$ would have to be of the form $B(J)=\alpha(J) \left( \begin{smallmatrix} 1 & 0 \\ 0 & -1 \end{smallmatrix} \right)$,
but now we can already see that only $\alpha\equiv 0$ works here since this $B$ leads to
\[
T(t;J) = \begin{pmatrix} e^{\omega(t;J)} & 0 \\ 0 & e^{-\omega(t;J)} \end{pmatrix}, \quad \omega(t;J) \equiv \int_0^t \alpha(s\cdot J)\, ds ,
\]
so $T(t)m=e^{2\omega}m$, but half line Jacobi $m$ functions satisfy $m_+(z)=-1/z + O(z^{-2})$ for large $z$, so Theorem \ref{T1.4}
now shows that we must have $\omega=0$. The same conclusion could have been obtained by looking at \eqref{zctoda} more closely.
So we are not getting a non-trivial flow for $N=0$.

Moving on to the case $N=1$ then, we again start out by comparing coefficients for the highest power $z^2$. This shows that
$B_1(1\cdot J)A_1(J)-A_1(J)B_1(J)=0$, and since $A_1=(1/a_1)\left( \begin{smallmatrix} 1 & 0 \\ 0 & 0 \end{smallmatrix} \right)$,
it follows that
\[
B_1(J) = \alpha_1 \begin{pmatrix} 1 & 0 \\ 0 & -1 \end{pmatrix} ,
\]
and here $\alpha_1(1\cdot J)=\alpha_1(J)$. Since there are Jacobi matrices $J$ whose orbits under the shift map are dense, this forces
us to take $\alpha_1$ as a constant, independent of $J$, at least if we want a continuous $B_1(J)$. In fact, since multiplying $B$ by a constant
amounts to rescaling time, we might as well set $\alpha_1=1$.

By comparing the coefficients of $z$, we obtain that
\begin{equation}
\label{3.3}
\dot{A_1} = B_1 A_0(J) - A_0(J) B_1 + B_0(1\cdot J) A_1(J) - A_1(J) B_0(J) .
\end{equation}
Since only the $(1,1)$ entry of $A_1$ is non-zero, this also gives restrictions on the possible choices for $B_0$. More specifically,
we find that
\begin{equation}
\label{3.2}
B_0(J) = \begin{pmatrix} \alpha_0(J) & 2 \\ -2a_0^2 & -\alpha_0(J) \end{pmatrix} .
\end{equation}
Finally, we compare the coefficients of $z^0$; this gives that
\begin{equation}
\label{3.4}
\dot{A_0} = B_0(1\cdot J) A_0(J) - A_0(J) B_0(J) .
\end{equation}
We work out the matrix elements, using the $B_0$ from \eqref{3.2}. From the $(1,2)$ element, we find that
$\dot{a}_1/a_1= -2b_1-\alpha_0(J)-\alpha_0(1\cdot J)$. On the other hand, comparing the $(1,1)$ elements of \eqref{3.3} yields
$\dot{a}_1/a_1 = \alpha_0(J)-\alpha_0(1\cdot J)$. So we must take $\alpha_0(J)=-b_1$. With these choices in place,
\eqref{3.4} will now produce the familiar classical Toda equations
\[
\dot{a}_n/a_n= b_{n+1}-b_n , \quad \dot{b}_n=2(a_n^2-a_{n-1}^2) ,
\]
initially for $n=1$, but then we also obtain the general case by considering $n\cdot J$ instead of $J$.
Moreover, and more importantly still for us perhaps, we have confirmed that
$B(J)$ for the Toda flow is indeed given by \eqref{Btoda}.

As I mentioned above,
Ong \cite{OngToda} has carried out this whole analysis in a systematic fashion, and he proves that the whole Toda hierarchy
can be reconstructed in this way.
\section{Proof of Theorems \ref{T1.2}, \ref{T1.3}}
\begin{proof}[Proof of Theorem \ref{T1.2}]
I'll use the following notations in this proof: $T=\left( \begin{smallmatrix} a & b \\ c & d \end{smallmatrix} \right)$,
and I'll write $m_{\pm}$ and $M_{\pm}$ for $m_{\pm}^{(1)}$ and $m_{\pm}^{(2)}$, respectively, and similarly
for other quantities, to be introduced in a moment. Let $h=m_++m_-$, $g_0=-1/h$, $g_1=m_+m_-/h$. Notice
that $g_0,g_1$ are the diagonal elements of $M$ from \eqref{defM}, so $\rho=\rho_0+\rho_1$ may be used as
a spectral measure, where $\rho_j$ denotes the measure associated with the Herglotz function $g_j$.

We thus want to show that $\rho$ and $P$ (``capital rho;'' recall our general convention of using uppercase letters
for the transformed quantities) are equivalent measures, and we also
need to pay attention to possible spectral multiplicity, but this latter part will be easy since only the absolutely continuous
part can have multiplicity greater than one.

By writing out the linear fractional transformation by which $T$ acts, we obtain that
\begin{equation}
\label{2.21}
M_+ = \frac{am_++b}{cm_++d}, \quad M_- = - \frac{am_--b}{cm_--d} ,
\end{equation}
and then that
\begin{equation}
\label{2.23}
G_0 = -(cm_++d)(cm_--d)g_0, \quad G_1 = -(am_+ +b)(am_- -b)g_0 .
\end{equation}

We'll treat the different parts of the spectrum separately, and we start with the absolutely continuous parts.
These are easy to analyze:
The corresponding half line spectral measures are given by
$d\rho_{\pm , \textrm{\rm ac}}(t)=(1/\pi)\textrm{Im}\: m_{\pm}(t)\, dt$, and we may look at the
operator of multiplication by the variable in $L^2(\R, d\rho_{+,ac})\oplus L^2(\R, d\rho_{-,ac})$; the fact that such
a direct reduction to the two half lines works is sometimes referred to as the \textit{decomposition method.}
Since $a,b,c,d$ are entire
functions, \eqref{2.21} will imply, after a quick calculation, that each of the two sets
\[
\Sigma_{ac}(m_{\pm}) := \{ t\in\R : \textrm{\rm Im}\: m_{\pm}(t)>0 \}
\]
differs from its counterpart for $M_{\pm}$ by at most a set of Lebesgue measure zero. This proves the claim about
the absolutely continuous parts, including multiplicity.

Since only the absolutely continuous part can have multiplicity greater than one, our discussion of the singular parts
can focus on showing that $\rho_s$ and $P_s$ are equivalent measures. Now $\rho_{0,s}$ is supported
by the set where $|g_0(t)|\equiv \lim_{y\to 0+}|g_0(t+iy)|=\infty$. Since $m_{\pm}$
are Herglotz functions, this is equivalent to the three conditions
\begin{equation}
\label{2.22}
\textrm{\rm Im}\: m_{\pm}(t) = 0, \quad \lim_{y\to 0+}\textrm{\rm Re}\: (m_+(t+iy) + m_-(t+iy))=0 .
\end{equation}
Now we are going to use Poltoratski's Theorem \cite{Polt} (see also \cite{JL}) on the comparison of the singular parts. This says that for
$\rho_{0,s}$-almost every $t$, the limit
\[
\lim_{y\to 0+} \frac{G_0(t+iy)}{g_0(t+iy)} \equiv L(t)
\]
exists, and $L(t)=dP_0(t)/d\rho_{0,s}$, the Radon-Nikodym derivative of that part of $P_0$ (or $P_{0,s}$) that is absolutely continuous
with respect to $\rho_{0,s}$.

I now want to know how frequently $L$ can be zero, to control how much of the measure $\rho_{0,s}$ could potentially
be lost after the transformation.
We may then restrict our attention to those $t$'s that satisfy \eqref{2.22} to compute $L$ because $\rho_{0,s}$ gives zero weight
to the set where these conditions fail.
We now distinguish two cases:

(a) $c(t)=0$: Then $c(t+iy)=O(y)$, and a Herglotz function $F$ always satisfies $\lim_{y\to 0+} y\textrm{Re}\: F(t+iy)=0$. Thus
also $\lim c(t+iy)m_+(t+iy)= 0$, since we are currently only considering
$t$'s with $\textrm{\rm Im}\: m_+(t)=0$. Hence $L(t)=d^2(t)$ and $c,d$ cannot both be zero since
$\det T=1$, so we never have $L=0$ in this case.

(b) $c(t)\not= 0$: Recall that we are currently considering only $t$'s satisfying \eqref{2.22}. Now if also $L(t)=0$ for such a $t$,
then $m_{\pm}(t)\equiv\lim_{y\to 0+} m_{\pm}(t+iy)$ both exist and
$m_-(t)=-m_+(t)=d(t)/c(t)$. It follows that
\[
\frac{G_1(t+iy)}{g_0(t+iy)} = -(am_++b)(am_--b) \to \left( -\frac{a(t)d(t)}{c(t)} + b(t) \right)^2 = \frac{1}{c^2(t)} .
\]
Thus Poltoratski's theorem shows that if $L(t)=0$ on a set of positive $\rho_{0,s}$-measure, then
$dP_{1,s}/d\rho_{0,s}>0$ for $\rho_{0,s}$-almost every such $t$.

Putting things together, we thus see that
$\rho_{0,s} \ll \left( P_0 + P_1\right)_s$; indeed, the argument we gave showed that when the
density $L(t)$ in the decomposition $dP_{0,s} = Ld\rho_{0,s}+d\nu$ fails to be positive, then $P_{1,s}$ comes to the rescue.

A similar analysis, with $-1/m_{\pm}$ now taking over the roles of $m_{\pm}$,
works for $\rho_{1,s}$, and thus it also follows that $\rho_s\ll P_s$.
By symmetry, since $T^{-1}\in\SL$ also, this then gives that $P_s\ll \rho_s$ as well, so $\rho_s$ and $P_s$ are
equivalent measures. This concludes the proof of the claim that $H_1,H_2$ are unitarily equivalent.

It remains to show that $|R_1|=|R_2|$ almost everywhere, but this is immediate from a calculation. Or, in flashier style,
one can observe that
\[
|R(z)| = \left| \frac{m_+(z)+\overline{m_-(z)}}{m_+(z)+m_-(z)}\right| = \tanh \left( \frac{1}{2} \gamma(m_+(z),-\overline{m_-(z)})\right)
\]
depends only on the hyperbolic distance $\gamma$ of the numbers $m_+(z)$, $-\overline{m_-(z)}\in \C^+$, which is preserved by the
automorphism $T(x)\in \textrm{\rm SL}(2,\R)$. Strictly speaking, this argument works only if $m_+(x),-\overline{m_-(x)}$
actually lie in $\C^+$, but if at least one of them is real, then $|R(x)|=1$, and we have the same situation for the transformed
operator by the unitary equivalence of the absolutely continuous parts, which we already proved.
\end{proof}
\begin{proof}[Proof of Theorem \ref{T1.3}]
As in the previous proof, we'll write $T=\left( \begin{smallmatrix} a & b \\ c & d \end{smallmatrix} \right)$.
Then, by writing out the assumption that $\pm m_{\pm} = T(\pm m_{\pm})$, we see that $m_+(z)$ and $-m_-(z)$ both solve the
quadratic equation
\begin{equation}
\label{4.1}
c(z)x^2+(d(z)-a(z))x-b(z) = 0 .
\end{equation}
Its solutions are
\[
x = \frac{a-d}{2c} \pm \frac{1}{2c}\sqrt{D^2 - 4}, \quad D\equiv a+d ,
\]
at least if $c(z)\not=0$. Observe that we cannot have $c\equiv 0$ here: if $z\in\C^+$, then $m_+(z)\not= -m_-(z)$, so \eqref{4.1}
must have two distinct solutions, and if $c=0$, then this forces $a=d= \pm 1$, $b=0$, so $T=\pm 1$, but we explicitly assumed
that $T$ is not identically equal to the identity matrix or its negative.

Recall from the previous proof how we can extract the spectral properties of the whole line operator from $m_{\pm}$: we let
$h=m_++m_-$, $g_0=-1/h$, $g_1=m_+m_-/h$, and then $\rho=\rho_0+\rho_1$ is a spectral measure of
maximal type, where the measures $\rho_0, \rho_1$ come from the Herglotz functions $g_0,g_1$.

In fact, this is needed only to study the singular part; the absolutely continuous part can be read off directly
by identifying the sets where $\textrm{Im}\:m_{\pm}(t)>0$. Here this happens
for $t\in\R$ precisely if $|D(t)|<2$, so we see that the absolutely continuous spectrum is
of (uniform) multiplicity $2$ and is essentially supported by
\[
\Sigma_{ac}=\sigma_{ac}=E = \{ t\in\R: |D(t)| \le 2 \} .
\]
We also obtain that $m_+(t)=-\overline{m_-(t)}$ on this set, so $H\in\mathcal R^C(E)$, as claimed.

To complete the proof of the theorem, we would just have to show that there is no spectrum outside $E$;
we will instead, as promised, give a complete treatment of the spectral properties, since this is easy to do.

It is already clear, from what we reviewed above, that there is no singular continuous spectrum;
it remains to prove that only the isolated points of $E$ are possible eigenvalues of $H$ . A calculation shows that
\begin{equation}
\label{4.3}
g_0 = -\frac{c}{\sqrt{D^2-4}}, \quad g_1 = \frac{b}{\sqrt{D^2-4}} ,
\end{equation}
and eigenvalues occur when
\begin{equation}
\label{4.2}
\lim_{y\to 0+} -iyg_j(t+iy) > 0
\end{equation}
for at least one of $j=0,1$. This can happen only when $D(t)=\pm 2$, and let's consider the case where $D(t)=2$
and \eqref{4.2} holds for $j=0$ at this $t$. Use the Taylor expansions
\[
c(t+iy) = c_k(iy)^k + O(y^{k+1}), \quad D(t+iy) = 2 + D_m(iy)^m + O(y^{m+1}) ,
\]
with $c_k,D_m\not= 0$ in \eqref{4.3}. It follows that $m=2n\ge 2$ must be even, and since $c_k,D_m\in\R$, we also see that $D_m>0$,
so indeed
\[
D(t+h)=2+D_m h^{2n} + O(h^{2n+1}) > 2
\]
for $x=t+h\in\R$ close to $t$, so $t$ is an isolated point of $E$, as claimed.
The other cases are similar.
\end{proof}
\section{Proof of Theorem \ref{T1.1}}
Before we can prove Theorem \ref{T1.1}, we of course need to give the precise definition of the cocycle:
The corresponding matrix function $B$ is given by
\begin{equation}
\label{todaB}
B(J) = \begin{pmatrix} 2(z-b_1)G_1 - H_1 & 2G_1 \\ -2a_0^2G_0 & -2(z-b_1)G_1 + H_1 \end{pmatrix} .
\end{equation}
This formula needs some explanation.
Also, recall that we have one such $B=B_p$ for each polynomial $p\in\mathcal P$, which, as usual,
is then used to produce $T(p;J)$ by solving
\[
\dot{T}=B_p(tp\cdot J)T, \quad T(0)=1 ,
\]
and evaluating at $t=1$.

For a given polynomial $p\in\mathcal P$,
the functions $G,H$ from \eqref{todaB} depend on $J\in\mathcal J$, $n\in\Z$, and $z\in\C$, and they
are polynomials in $z$; the indices in \eqref{todaB} refer to the $n$ variable, so for example $G_1=G(n=1)$.
The precise definitions are as follows: Suppose first of all that $p(x)=x^k$, and attach a superscript $k$ to the
corresponding functions $G,H$, as a reminder to ourselves that this is the $p$ we have currently chosen. Then
\begin{align*}
G^{(k)}(z,J) & = \sum_{j=0}^{k-1} z^{k-1-j} (J^j)_d , \\
H^{(k)}(z,J) & = z^k - (J^k)_d + 2a \sum_{j=1}^{k-1} z^{k-1-j} (SJ^j)_d ;
\end{align*}
here $X_d$ refers to the diagonal part of a (let's say: bounded) operator, thought of as an infinite matrix. In other words,
$(X_d)_n = \langle \delta_n, X \delta_n\rangle$. As before, $S$ denotes the shift,
so $(Sx)_n=x_{n+1}$ for $x\in\ell^2$, and $a$ refers
to the sequence $(a_n)$. So the dependence on $n$ has not been made explicit
in the notation; it is contained in the various diagonal parts and also in $a=a_n$.

We have now defined the functions $G,H$ for $p(x)=x^k$;
for general $p(x)=\sum c_k x^k$, we put $G=G_p=\sum c_k G^{(k)}$, and similarly for $H=H_p$. In other words,
we make $G,H$ linear functions of $p$.

If $p(x)=x$, then we obtain that $G=1$, $H=z-J_d=z-b$, and thus we recover the $B$
from \eqref{Btoda} from \eqref{todaB}.

We are now ready for the
\begin{proof}[Proof of Theorem \ref{T1.1}]
I will only discuss the cocycle identity for evolution along two distinct flows
from the Toda hierarchy. The theorem also claims that the same property holds if we consider a flow together with the shift.
This follows from the zero curvature equation, and the relevant calculations (though not the statement itself, at least not
very explicitly) can be found in the standard literature,
for example in \cite[Section 12.2]{Teschl} or \cite[Section 1.3]{GeHol2}. The argument is presented in detail
in the appendix of \cite{OngR}.

So we'll show that
\begin{equation}
\label{5.32}
T(p+q;J)=T(p; q\cdot J)T(q; J)\quad \textrm{for }p,q\in\mathcal P .
\end{equation}
Rather than deal with this $\SL$-cocycle directly,
we can reduce matters to the same question about a $\C^{\times}$ valued cocycle, as follows. Introduce,
for fixed $z\in\C^+$,
\[
V(J) = \frac{1}{\sqrt{m_++m_-}} \begin{pmatrix} m_+ & - m_- \\ 1 & 1 \end{pmatrix} ;
\]
the choice of square root can be conveniently settled by requiring that $\sqrt{m_++m_-}\in\C^+$.
Next, write
\begin{equation}
\label{5.31}
T(p;J) = V(p\cdot J) D(p;J) V(J)^{-1}, \quad D = \begin{pmatrix} \lambda(p;J) & 0 \\ 0 & \lambda(p;J)^{-1} \end{pmatrix} .
\end{equation}
(This kind of transformation is usually described by referring to $T$ and $D$ as \textit{cohomologous }cocycles; note, however,
that $T\in\SL$ while $V$ is not an entire function of $z$. So the transformation, while convenient here, certainly has its drawbacks
from a general point of view.) To interpret \eqref{5.31}, let's first of all observe that $V$ as a linear fractional transformation
maps the standard unit vectors $e_1$, $e_2$ (that is, $\infty$ and $0$, as points on the Riemann sphere) to $\pm m_{\pm}$. This means that
\eqref{5.31} simply records the general form of an $\textrm{SL}(2,\C)$ matrix that updates $\pm m_{\pm}(J)$ correctly to
their new values $\pm m_{\pm}(p\cdot J)$. Next, the cocycle property \eqref{5.32} for $T$ is now equivalent to the same
property for $D$, and since $D$ is diagonal, \eqref{5.32} indeed simplifies to
\begin{equation}
\label{7.2}
\lambda(p+q;J) = \lambda(p; q\cdot J)\lambda(q; J) ,
\end{equation}
and this is what we'll now establish.

Since we already have the cocycle property of $D$ and $\lambda$ with respect to the action of $G=\R$, for
just a fixed individual flow, we know that, in analogy to \eqref{todacc}, we can obtain $\lambda$ by solving an ODE
of the form
\begin{equation}
\label{lambdacc}
\dot{\lambda} = \omega_p(tp\cdot J) \lambda, \quad \lambda(0)=1 ;
\end{equation}
as usual, $\lambda(p;J)$ then is the solution to this initial value problem, evaluated at $t=1$.

We can find $\omega=\omega_p(J)$ by comparing this with \eqref{todacc}, \eqref{todaB}. To do this,
differentiate $T=T(tp;J)$ with respect to $t$ at $t=0$ to produce $B_p(J)$. Obviously, these manipulations
assume that the various evolving quantities depend smoothly on $t$; a few quick remarks about this issue
can be found below, at the beginning of the proof of Proposition \ref{P7.1}.

If the representation from
the right-hand side of \eqref{5.31} is used in this calculation, then we obtain an alternative formula that relates $B$ to
$\omega_p(J)=(d/dt)\lambda(tp; J)\bigr|_{t=0}$. It is especially convenient to focus on the $(2,1)$ entry
of $B$ because the only contribution to this when differentiating \eqref{5.31} comes from
$V(J)\dot{D}V(J)^{-1}$. In this way, we see that
\begin{equation}
\label{defom}
\omega_p(J) = -a_0^2 (m_++m_-) G_0 = (G/g)_0 ,
\end{equation}
where $g_0 = -1/(a_0^2(m_++m_-))$. This is also the matrix element
of the resolvent
\begin{equation}
\label{defg}
g_0(z)=\langle \delta_0, (J-z)^{-1} \delta_0\rangle ,
\end{equation}
which we already used
in Section 4 (with the factor $a_0^2$ removed). So alternatively, we could define $g = (J-z)^{-1}_d$ and
then interpret $g_0$ exactly as above as the $n=0$ element of this sequence.

The choice of notation here was deliberate: the polynomial $G$ is a truncated version of the Taylor series of $g$
about $z=\infty$, plus a trivial shift of exponent. So the two functions $g,G$ are closely related. There is an analog $h$ of $g$,
which bears the same relation to $H$ as $g$ does to $G$ and will become important later. It is defined by
\begin{equation}
\label{defh}
h(z) = \left( 2aS(J-z)^{-1} -1 \right)_d ;
\end{equation}
here, $S$ again denotes the shift operator $(Sx)_n=x_{n+1}$, and $a$ must now be interpreted
as the operator of multiplication by the sequence $a_n$.

We will now establish:
\begin{Proposition}
\label{P7.1}
For fixed $z\in\C^+$ and
any two polynomials $p,q\in\mathcal P$, $\omega_p(tq\cdot J)$ is a smooth function of $t\in\R$ and
\begin{equation}
\label{7.29}
\frac{d}{dt} \omega_p(tq\cdot J)\bigr|_{t=0} = \frac{d}{dt} \omega_q (tp\cdot J)\bigr|_{t=0} .
\end{equation}
\end{Proposition}
Before proving this, let's discuss how we can obtain \eqref{7.2} (and thus finish the proof of Theorem \ref{T1.1}) from
Proposition \ref{P7.1}. By applying this result to $(sp+tq)\cdot J$ in place of $J$, we also obtain that
\begin{equation}
\label{7.4}
\frac{\partial}{\partial s} \omega_q((sp+tq)\cdot J) = \frac{\partial}{\partial t} \omega_p ((sp+tq)\cdot J) ,
\end{equation}
and then integration with respect to $s$ from $s=0$ to $s=t$ yields
\begin{equation}
\label{7.4c}
\omega_q(t(p+q)\cdot J) - \omega_q(tq\cdot J) = \int_0^t \frac{\partial}{\partial t} \omega_p ((sp+tq)\cdot J)\, ds .
\end{equation}
We also have that $\omega_{p+q}=\omega_p+\omega_q$. To see this, recall how $\omega$ was defined in \eqref{defom}:
only $G_0$ depends on $p$ here, and we observed earlier that $G$ indeed is a linear function of $p$. So we can rewrite
\eqref{7.4c} as
\begin{align}
\label{7.5}
\omega_{p+q}(t(p+q)\cdot J) = &\:\: \omega_p(t(p+q)\cdot J) + \omega_q(tq\cdot J) \\
\nonumber
& + \int_0^t \frac{\partial}{\partial t} \omega_p ((sp+tq)\cdot J)\, ds .
\end{align}
Next, integrate this from $t=0$ to $t=1$ and use Fubini-Tonelli on the right-hand side, which is justified
since $(sq+tp)\cdot J$ stays inside a compact set for $s,t\in [0,1]$,
and this implies that the integrand is bounded. We do need further information about
the evolution under Toda flows for this step, which can be found in \cite{Teschl}; see also the discussion below,
at the beginning of the proof of Proposition \ref{P7.1}. Integration of the last term of \eqref{7.5} then produces
\begin{gather*}
\int_0^1 ds\int_s^1 dt\, \frac{\partial}{\partial t} \omega_p ((sp+tq)\cdot J) = \\
\int_0^1 \left( \omega_p(sp\cdot (q\cdot J))
- \omega_p(s(p+q)\cdot J)\right) \, ds
\end{gather*}
and here the last term conveniently cancels the integral of the first term from the right-hand side of \eqref{7.5}. So if we also recall
\eqref{lambdacc} and exponentiate the integrated version of \eqref{7.5}, then we indeed arrive at \eqref{7.2}, as desired. This
completes the proof of Theorem \ref{T1.1}, assuming Proposition \ref{P7.1}.
\end{proof}

\begin{proof}[Proof of Proposition \ref{P7.1}]
As for the smoothness claims, we essentially refer the reader to \cite{Teschl} and limit ourselves to a few remarks.
See especially Theorem 12.6 there, but also Lemma 12.15. As a general strategy, once the smoothness of (the coefficients of)
$tp\cdot J$ is known, everything else (such as the smoothness of $m_{\pm}(tp\cdot J)$, $G_0(tp\cdot J)$) will pretty much just
fall into place by making use of the explicit formulae we have; the perhaps most challenging contributions here are $m_{\pm}$,
but these, too, can be handled without much trouble by recalling that they are evolved by the cocycle and then using the explicit
formulae for $B$. I'll leave the matter at that, and I'll focus on \eqref{7.29} now.

First of all, observe that it suffices to prove this for monomials $p(x)=x^p$, $q(x)=x^q$. Indeed, suppose we had this already,
and consider general polynomials $p=\sum p_jx^j$, $q=\sum q_j x^j$. Then
\begin{gather*}
\omega_p(tq\cdot J) = F(t,t, \ldots, t), \\
F(t_0, \ldots , t_N) \equiv \omega_p(t_0 q_0 \cdot t_1 q_1 x \cdot \ldots \cdot t_Nq_Nx^N \cdot J) ,
\end{gather*}
and since the group acting is abelian, the monomials $t_jq_jx^j$ can be reshuffled at will here. Thus, by the chain rule,
\[
\frac{d}{dt} \omega_p(tq\cdot J)\bigr|_{t=0} = \sum \frac{\partial F}{\partial t_j}(0,0, \ldots , 0) = \sum \frac{d}{dt}
\omega_p (tq_j x^j \cdot J) \bigr|_{t=0} .
\]
Moreover, $\omega_p = \sum p_j \omega_{x^j}$, so indeed \eqref{7.29} for monomials gives the general case also.

So it now suffices to take
$p(x)=x^p$, $q(x)=x^q$, and let's also assume that $p>q$. We'll write $G^{(p)}, G^{(q)}, H^{(p)}, H^{(q)}$ to indicate which polynomial
(or, rather, monomial) is being used, and we now adopt the convention that evaluation at $n=0$ is understood if no such index is given:
for example $G^{(p)}$ will refer to $(G^{(p)})_0$.

We will also need the time evolution of $g=g(tq\cdot J)$, which is given by
\begin{equation}
\label{evolg}
\dot{g} = 2g (\omega_q h - H^{(q)}) ;
\end{equation}
here, as always, the dot notation is defined as $\dot{X}(J) = (d/dt)X(tq\cdot J)\bigr|_{t=0}$, and $h$ was defined in \eqref{defh}.
To prove \eqref{evolg}, it is easiest to work with the explicit kernel of the Green function $g$; please see \cite[Section 12.4]{Teschl}
for the details.

It now follows that
\begin{equation}
\label{7.81}
\frac{d}{dt} \omega_p(tq\cdot J) \bigr|_{t=0} = \frac{1}{g} \left( \dot{G}^{(p)} - 2\omega_q h G^{(p)} + 2H^{(q)} G^{(p)} \right) .
\end{equation}
To analyze this further, we expand $g,h$ about $z=\infty$:
\begin{equation}
\label{expgh}
g(z) = -\sum_{n=0}^{\infty} c_n z^{-n-1}, \quad h(z) = - \sum_{n=-1}^{\infty} d_n z^{-n-1} ,
\end{equation}
and here
\begin{gather*}
c_0=1, \: c_n = \langle \delta_0, J^n \delta_0\rangle \quad (n\ge 1); \\
d_{-1}=1, \: d_0=0, \: d_n=\langle \delta_0, 2aSJ^n \delta_0 \rangle \quad (n\ge 1) .
\end{gather*}
These formulae follow by expanding the definitions \eqref{defg}, \eqref{defh}, and the series converge at least
for $|z|>R$, with $R=\|J\|$. Note that this quantity does not change under any Toda flow because the evolved operator
is unitarily equivalent to $J$.

It is again clear from \cite[Theorem 12.6]{Teschl} that $c_n ,d_n\in C^{\infty}$ as a function of $t$ along any flow. Moreover,
we also obtain uniform bounds of the form $\dot{c_n}, \dot{d_n}\lesssim R^n$, so the expansions from
\eqref{expgh} may be differentiated with respect to $t$ term by term, and thus $\dot{g},\dot{h}$ have expansions similar
to the ones from \eqref{expgh}. This will become important in a moment because we will establish the desired identity
by expanding and then comparing coefficients. Note that strictly speaking this will only give the identity for $|z|>R$, but of course
that is good enough since both sides are holomorphic functions of $z$.

As I already pointed out, the functions $g,h$ are closely related to $G,H$, and this can now be made more explicit: we have that
\[
G^{(D)}(z) = z^D \sum_{n=0}^{D-1} c_n z^{-n-1} , \quad H^{(D)}(z) = z^D \sum_{n=-1}^{D-1} d_n z^{-n-1} - c_D, \quad D=p,q .
\]
In other words, we have that (for example)
\[
G^{(p)}(z) = \left[ -z^pg(z) \right]_+ ,
\]
where the notation $[X]_+$ instructs us to expand $X=\sum_{n\in\mathbb Z} X_nz^n $ into a Laurent series
about $z_0=0$ and then keep only the power series part $[X]_+=\sum_{n\ge 0}X_nz^n$ (which, in this paper,
will always be a polynomial).
We now use these formulae to compute the time derivatives. So, recalling \eqref{defom} and \eqref{evolg}, we can now say that
\begin{equation}
\label{7.4d}
\dot{G}^{(p)} = \left[ -2z^p (G^{(q)} h- H^{(q)} g) \right]_+ .
\end{equation}
In this step, we make use of our preliminary observations about term-by-term differentiation of \eqref{expgh}.
If \eqref{7.4d} is combined with \eqref{7.81}, then what we are trying to show can be rephrased as the claim that with
\[
F(p,q)\equiv \left[ -z^p (G^{(q)} h- H^{(q)} g) \right]_+ - \omega_q h G^{(p)} + H^{(q)} G^{(p)} ,
\]
we have the identity $F(p,q)=F(q,p)$.

This symmetry property is clear for the second term if we recall \eqref{defom}, so we may drop this from $F$ and focus on
\[
F_1(p,q) \equiv\left[ H^{(q)} G^{(p)} -z^p (G^{(q)} h- H^{(q)} g) \right]_+ .
\]
We can now finish the argument by a straightforward (if tedious) brute force calculation.
We expand everything in powers of $z$. After a calculation, we find that

\begin{multline*}
F_1(p,q) = \left[ z^{p+q} \left( \sum_{j=0}^{q-1} c_j z^{-j-1}\sum_{k\ge -1} d_k z^{-k-1} \right. \right. \\
\left. \left.
- \sum_{j\ge p} c_j z^{-j-1}\sum_{k= -1}^{q-1} d_k z^{-k-1}\right) \right]_+ .
\end{multline*}
In the last sum, we may also sum over \textit{all }$k\ge -1$ since we will not get contributions to the $[\ldots ]_+$ part from
the $k\ge q$, and if written in this form, then $F_1$ is easily seen to have the symmetry property we require.
\end{proof}

\section{Twisted shifts and flows for canonical systems}
As pointed out earlier, though not exceedingly popular at the moment, \textit{canonical systems}
\begin{equation}
\label{can}
Ju'(x) = -zH(x)u(x) , \quad x\in\R ,
\end{equation}
are a very natural object from a mathematical point of view because when we normalize them by
the requirement that $\textrm{tr}\: H = 1$, then they are in one-to-one correspondence to pairs
of Herglotz functions $m_{\pm}: \C^+\to\overline{\C^+}$. Here $J$ denotes the matrix
$J=\left( \begin{smallmatrix} 0 & -1 \\ 1 & 0 \end{smallmatrix} \right)$, and the basic assumptions
on the coefficient function $H$ are
as follows: it takes values in $\R^{2\times 2}$,
$H(x)\ge 0$ almost everywhere, and $H\in L^1_{\textrm {loc}}$ (this latter requirement follows
automatically if the entries are measurable and we do normalize the trace). The half line $m$ functions are
defined as $m_{\pm}(z) = \pm f_{\pm}(0,z)$, $z\in\C^+$, where $f_{\pm}$ solves \eqref{can} and
$f_{\pm}\in L^2_H(\R_{\pm})$, that is, $\int_0^{\infty} f^*_+(x)H(x)f_+(x)\, dx<\infty$, and similarly for $f_-$.
Here and throughout this section,
we assume limit point case at both endpoints $\pm \infty$: there is a unique, up to a multiplicative constant,
square integrable solution at each endpoint. If $\textrm{tr}\: H=1$, or, more generally, $\textrm{tr}\: H\notin L^1(\R_{\pm})$,
then this follows automatically, by a Theorem
of de~Branges, see \cite{dB} and also \cite{AchaLPC}. Finally, when we write $m_{\pm}(z)=\pm f_{\pm}(0,z)$ we have again
used the convention that a non-zero vector $v\in\C^2$ is identified with the point $v_1/v_2$ on the Riemann sphere.

Since any Herglotz function is the $m$ function of a (unique, when trace-normed) canonical system (on a half line),
it must in particular be possible to rewrite Jacobi and Schr\"odinger equations
\begin{equation}
\label{se}
-y''(x) + V(x)y(x) = z y(x)
\end{equation}
as canonical systems. This can be done explicitly, and these transformations are well known. Let me discuss the
case of a Schr\"odinger equation \eqref{se}, with potential $V\in L^1_{\textrm {loc}}(\R)$,
and in the limit point case at $\pm\infty$. Given a solution $y$, let $Y=(y',y)^t$, and observe that $Y$ then solves
\[
Y' = \begin{pmatrix} 0 & V-z \\ 1 & 0 \end{pmatrix} Y .
\]
Let $T_0(x)\in\textrm{\rm SL}(2,\R)$ be the matrix solution of this equation
for $z=0$ and with the initial value $T_0(0)=1$. Write
\[
T_0(x) = \begin{pmatrix} p'(x) & q'(x) \\ p(x) & q(x) \end{pmatrix}.
\]
Finally, introduce $u(x)$ by writing $Y=T_0u$. This is essentially variation of constants for \eqref{se} about $z=0$,
which seems a reasonable thing to try since we must get rid of the terms not involving $z$ if we want to write
\eqref{se} as a canonical system. A calculation then shows that: (1) $u$ indeed solves \eqref{can}, with
\begin{equation}
\label{canse}
H(x)= \begin{pmatrix} p^2(x) & p(x)q(x) \\ p(x)q(x) & q^2(x) \end{pmatrix} ;
\end{equation}
(2) $y\in L^2(0,\infty)$ precisely if $u\in L^2_H(0,\infty)$, and since the $m$ function of \eqref{se}
with Dirichlet boundary conditions $y(0)=0$ is given by $m_S(z)=F(0,z)=f'(0,z)/f(0,z)$,
with $f\in L^2(0,\infty)$, this in particular shows that, as intended, $m_S=m_C$, that is, the $m$ functions of the
Schr\"odinger equation and the canonical system agree. The same remarks apply to
the left half line $(-\infty, 0)$.

Note that $H$ from \eqref{canse} will usually not be trace normed; we could pass to a trace normed version of this (or any) $H$
by the change of variable $t=\int_0^x \textrm{tr}\: H(s)\, ds$, but we prefer not to do so here. One reason for this
is that such a normalization is inconvenient here because it will typically not be preserved by the flows we are about to construct.

The classical hierarchy of evolution equations of Schr\"odinger operators is the KdV hierarchy. This can be constructed in
the traditional way from a Lax equation, or, more in line with what we do here,
in the same way as discussed in Section 3 for the Toda hierarchy: $\R$ acts by shifts $(s\cdot V)(x) = V(x+s)$
(it is now $\R$ rather than $\Z$ because we have moved from the discrete to the continuous setting),
and we may now look for new flows (= actions of $\R$) that come with associated $\SL$-cocycles for the
action of $G=\R\times\R$, where, as above, the extra copy of $\R$ acts by shifts. There are definitely extra
technical difficulties involved, compared to the Jacobi case, since we cannot expect global flows on very general
initial conditions, but leaving that aside, one finds that at least the formal side of this works just like before.

Moving on to canonical systems then, this would suggest to just try the same thing here, but I will instead
propose a modified and more general approach. I will not look for flows that commute with literally the shift $H(x)\mapsto H(x+s)$,
but rather with a modified version of this. One strong immediate motivation for this comes from the observation that
a shift $V(x)\mapsto V(x+s)$ will not induce just a shift on the corresponding $H$ from \eqref{canse}; in fact, since always
$H(0)=\left( \begin{smallmatrix} 0 & 0 \\ 0 & 1 \end{smallmatrix} \right)$ that clearly could not be true. Rather, the shift on $V$ induces
the action
\begin{equation}
\label{5.1}
(s\cdot H)(x) = T_0(s)^{-1t}H(x+s)T_0(s)^{-1}
\end{equation}
on $H$. This can be checked by hand, by verifying that the right-hand side is still \eqref{canse}, but with $T_0(x)$ replaced by
$T_0(x+s)T_0(s)^{-1}$ (this is the right way to do it because we need the solution that is the identity matrix at $x=0$), or,
alternatively, it could be checked that $m_{\pm}$ get updated correctly, though this second argument would require a uniqueness
result also.

Now we take this as our guideline for suitable flows on general canonical systems that might take over the role of the shift.
Motivated by \eqref{5.1}, I propose to consider $\R$ actions of the form
\begin{equation}
\label{5.2}
(s\cdot H)(x) = M(s;H)^tH(x+s)M(s;H) ,
\end{equation}
for functions $M:\R\times \mathcal C\to \textrm{SL}(2,\R)$ (where I denoted the set of coefficient functions of
limit point case canonical systems by $\mathcal C$) satisfying the following condition.
\begin{Proposition}
\label{P5.1}
Suppose that $T:=M^{-1}$ satisfies the cocycle identity
\[
T(s+t; H) = T(s; t\cdot H) T(t; H) .
\]
Then \eqref{5.2} defines a flow on $\mathcal C$.
\end{Proposition}
\begin{proof}
This is trivial; check it by direct calculation.
\end{proof}
We will call such an action of $G=\R$ on $\mathcal C$ a \textit{twisted shift.}

As we discussed earlier, (differentiable) cocycles correspond to matrix functions $S=S(H)$, $\textrm{tr}\: S=0$
(we used to denote these by $B$, but I now want to call them $S$ as in \textit{shift}): given such an $S$, the cocycle is
then obtained from
\begin{equation}
\label{5.4}
\dot{T} = S(s\cdot H) T, \quad T(0)=1 .
\end{equation}
Moreover, once we have the cocycle $T$, we then obtain a flow from \eqref{5.2}. However, notice that
this does not really give an explicit construction of group actions of $\R$, after having chosen a matrix function $S$
because in order to be able to solve \eqref{5.4}, we need to have that action already. In fact, a random choice of $S$ seems
quite unlikely to produce a group action via \eqref{5.2}, \eqref{5.4}, for reasons that will be much clearer in the discrete
setting, so we postpone the more detailed discussion until the end of this paper. Please see \eqref{ts1}, \eqref{ts2} below
and the comments that follow.

This somewhat circular structure can be camouflaged if we pass to a differential formulation. (However, note that it becomes
truly circular only if we start out at the wrong end, with a matrix function $S$; if we are given $M=M(s;H)$ instead, then this
matrix function satisfies the condition of Proposition \ref{P5.1} or it doesn't, and when it does, then we do obtain an action
and a cocycle. There is no circularity here because if $M$ is given, then \eqref{5.2} gives a perfectly meaningful definition
of $s\cdot H$, though, despite the notation, this map might fail to define a group action.)
To do this, assume that everything on the right-hand side of \eqref{5.2} is differentiable and take the $s$ derivative at $s=0$.
Since $(d/ds)M = -T^{-1}(dT/ds)T^{-1}$ and $T(0)=1$, $(dT/ds)(0)=S(H)$, this yields
\begin{equation}
\label{pdets}
\frac{\partial H}{\partial s} - \frac{\partial H}{\partial x} = -S^t(H)H - HS(H) ,
\end{equation}
where $H=H(x,s)=(s\cdot H)(x)$. More precisely, we initially obtain this equation for $s=0$, but then the general case, at $s=s_0$, say,
results by applying this to $(s\cdot (s_0\cdot H))(x)$.

We could now make this equation our starting point. Having chosen a matrix function $S=S(H)\in\R^{2\times 2}$, $\textrm{tr}\: S=0$,
we hope to obtain an action of $\R$ on (parts of) $\mathcal C$
from \eqref{pdets}. Of course, whether or not that will actually be the case would need further investigation in any given case; \eqref{pdets}
is a non-linear, non-local (thanks to the arbitrary dependence of $S$ on $\{H(x): x\in\R \}$) PDE, so in principle anything
could happen, and, as I just pointed out, I'm not particularly optimistic about our prospects if $S$ is just a randomly chosen
matrix function.

If we make the simplest possible choice, $S\equiv 0$, then \eqref{pdets} becomes
$\partial_s H-\partial_x H=0$, so recovers the plain shift $(s\cdot H)(x)=H(x+s)$.
A second important and more interesting choice of $S$ is given by
\[
S(H) = \begin{pmatrix} 0 & V(0) \\ 1 & 0 \end{pmatrix}, \quad V(x)\equiv \frac{1}{4}\det H''(x) .
\]
This extends the twisted shift that corresponds to the shift $V(x)\mapsto V(x+s)$ on Schr\"odinger operators to general canonical systems
(with a twice differentiable coefficient function $H$). To see that this is the case, check by a computation
that if $H$ corresponds to a Schr\"odinger operator
in the way explained above, so is of the form \eqref{canse}, then indeed $\det H'' = 4V$, with now $V=p''/p=q''/q$ being the potential from the
original Schr\"odinger operator.

In general, for an $S$ of this general type, we \textit{can }guarantee existence of global solutions to \eqref{pdets} and
existence of a group action.
\begin{Proposition}
\label{P5.2}
Suppose that $S=S(H)$ is of the form
\begin{equation}
\label{schrS}
S=F(\det H(0), \det H'(0), \ldots , \det H^{(n)}(0)) ,
\end{equation}
for some function $F\in C^1$, $\textrm{\rm tr}\: F=0$.
Then \eqref{pdets} has a global classical solution for any
initial value $H(x,s=0)=H_0(x)\in C^n(\R)$. Moreover, \eqref{5.2}, \eqref{5.4} yield an $\R$ action on coefficient functions
$H\in C^n(\R)$.
\end{Proposition}
To motivate the proof, let me mention the following fact, which is also of some independent interest. It says that the determinant
of $H$ is preserved along the characteristics $x+s=c$.
\begin{Proposition}
\label{P5.3}
Suppose that $H(x,s)$ solves \eqref{pdets}. Then $\det H(x,s) = \det H(x+s, 0)$.
\end{Proposition}
\begin{proof}
This is unsurprising since \eqref{pdets} was meant to be a rewriting of \eqref{5.2}, and $\det M=1$.
We'll check it by a calculation. It is helpful to introduce the new variables $u=s-x$, $v=s+x$. Then $2\partial H/\partial u = -S^tH-HS$,
and we want to show that $(\partial/\partial u) \det H=0$. Write
\[
H = \begin{pmatrix} a & b \\ b & c \end{pmatrix} , \quad S = \begin{pmatrix} \alpha & \beta \\ \gamma & -\alpha \end{pmatrix} .
\]
Then, with $X' \equiv \partial X/\partial u$,
\begin{gather*}
a' = - \alpha a - \gamma b ,\\
2b' = -\beta a - \gamma c ,\\
c' = \alpha c - \beta b ,
\end{gather*}
and plugging this into $(\det H)' = a'c + ac' -2bb'$ will show that indeed $(\det H)'=0$, as claimed.
\end{proof}
\begin{proof}[Proof of Proposition \ref{P5.2}]
Motivated by Proposition \ref{P5.3}, we make the following attempt: Given an initial value $H_0(x)$ for $H$,
define $T(s)$ by solving
\begin{equation}
\label{5.6}
\dot{T}(s) = F(\det H_0(s), \det H'_0(s), \ldots , \det H^{(n)}_0(s)) T(s), \quad T(0)=1 .
\end{equation}
If we compare this with \eqref{5.4}, then we find that this is not exactly the ``correct'' way to compute $T$,
since we've replaced the (unknown, at this point) action $(s\cdot H_0)(x)$ by $H_0(x+s)$;
however, by Proposition \ref{P5.3}, we expect that this
will have no consequences as far as computation of the determinants is concerned. Indeed, if we now define
\begin{equation}
\label{5.5}
H(x,s) = T(s)^{-1t}H_0(x+s) T(s)^{-1} ,
\end{equation}
then it is in fact obvious that $\det H_0^{(j)}(x+s) = \det (\partial^j H/\partial x^j)(x,s)$. Moreover, $(d/ds)T^{-1} = -T^{-1}S$,
by \eqref{5.6}, and now a computation shows that $H(x,s)$ from \eqref{5.5} solves \eqref{pdets}.

The \textit{moreover }part has the same proof: define $T$ by \eqref{5.6} and then check that this $T$ works in \eqref{5.2}, \eqref{5.4}.
\end{proof}

The next result completes the analogy between plain and twisted shifts by providing an $\SL$-cocycle for
twisted shifts also, which again updates $m$ functions correctly. Notice that from the point of view of the
cocycles, the twisted shift
is a plain shift, followed by the action of an $\textrm{\rm SL}(2,\R)$ matrix (in other words, an automorphism
of $\C^+$).
\begin{Theorem}
\label{T5.1}
Consider the action \eqref{5.2} of $G=\R$, and denote the associated $\textrm{\rm SL}(2,\R)$-cocycle from Proposition \ref{P5.1}
by $T_0=T_0(s;H)=M^{-1}$.
Let $T_1=T_1(x;H)$ be the $\SL$-cocyle of the shift: $dT_1/dx =zJHT_1$, $T_1(0)=1$.
Then $T(s;H)=T_0(s;H)T_1(s;H)$ defines an $\SL$-cocycle for this action, and
\begin{equation}
\label{5.9}
\pm m_{\pm}(z; s\cdot H) = T(s;H) (\pm m_{\pm}(z;H)) .
\end{equation}
\end{Theorem}
\begin{proof}
Let's first recall the following well known fact:
If $A\in\textrm{SL}(2,\R)$ and $H_A(x):=A^t H(x) A$, then
$\pm m_{\pm}(z; H_A) = A^{-1}(\pm m_{\pm}(z;H))$. This is easy to verify by looking at how solutions
are changed under this transformation.

Since $(s\cdot H)(x) = T_0^{-1t}(s; H) H(x+s) T_0^{-1}(s; H)$ and $T_1$ updates $m_{\pm}$ under the shift flow,
this in particular implies \eqref{5.9}. So we only need to verify that $T$ is a cocycle. Since $T_0,T_1$ are cocycles
individually, we have that
\begin{align*}
T(s+t;H)& = T_0(s+t;H) T_1(s+t; H) \\
& = T_0(s; t\cdot H) T_0(t;H) T_1(s; S_tH) T_1(t; H)
\end{align*}
(writing $(S_tH)(x)\equiv H(x+t)$),
so this will follow if we can show that
\begin{equation}
\label{5.11}
T_0(t;H) T_1(s; S_tH) = T_1(s; t\cdot H) T_0(t; H) .
\end{equation}
Call the two sides of this equation $A$ and $B$, respectively, and observe that $A,B$ are absolutely continuous
functions of $s$ (even if $T$ itself isn't, which might happen if we choose a sufficiently irregular $T_0$). We compute
\[
\frac{dA}{ds} = z T_0(t;H)JH(s+t)T_1(s; S_t H) = zJ (t\cdot H)(s) A ;
\]
in the last step we have used the identity $CJC^t=J$, which is valid for $C\in\textrm{\rm SL}(2,\R)$. Similarly,
\[
\frac{dB}{ds} = zJ(t\cdot H)(s) B ,
\]
so $A,B$ solve the same linear ODE, and
since also $A=B$ at $s=0$, we indeed obtain \eqref{5.11}.
\end{proof}
Since $T=T_0T_1$ is a cocycle, if it is sufficiently regular, it must then itself satisfy an equation of the type
\[
\dot{T} = C(s\cdot H) T, \quad T(0)=1 ,
\]
and indeed this is obviously the case, with
\begin{equation}
\label{tsccC}
C(H)=\frac{d}{ds}\,(T_0T_1)\bigr|_{s=0} = S(H) +zJH(0) .
\end{equation}
This concludes our discussion of twisted shifts.

Let me now take a quick look at what could be done with this, following the ideas from Sections 2, 3,
without attempting a systematic discussion.
Let's first try to give a version of Theorem \ref{T1.4} that applies to the new setting.
Here, we will make essential use of the \textit{Weyl disks}
\[
\mathcal D(x,z;H) = T_1(x,z;H)^{-1}\overline{\C^+} \quad (x>0,z\in\C^+) ;
\]
as above, $T_1$ denotes the cocycle for the shift, so $T_1'=zJHT_1$, $T_1(0)=1$, and the closure of $\C^+$ is taken on the Riemann sphere,
so includes $\infty$. This set $\mathcal D(x,z;H)$ is indeed a disk; its boundary is the image of $\R_{\infty}$
under the linear fractional transformation $T_1^{-1}$.
It is in fact a disk that is contained in $\overline{\mathbb C^{+}}$ itself. From the cocycle property it is obvious that these disks are nested as $x$
increases: $\mathcal D(x,z;H) \subseteq \mathcal D(y,z;H)$ if $x\ge y$. So as $x\to\infty$, the disks $\mathcal D(x,z;H)$ shrink to a limiting
object $\bigcap_{x>0} \mathcal D(x,z;H)$, and we are assuming limit point case here, which means that this will be a point. More precisely,
the intersection of the Weyl disks will be the unique point $m_+(z;H)$.

It's also useful to observe that $\mathcal D(x,z;H)$ can be directly understood as those values that $m_+(z;H)$ can still take, given the
values of $H(t)$ on $0\le t\le x$. This is an immediate consequence of the fact that $T_1$ updates $m_+$ when $H$ is shifted, and if
we don't know anything about a half line, then any (generalized) Herglotz function is possible as its $m_+$.
\begin{Theorem}
\label{T5.2}
Fix $z\in\C^+$, and
consider an action of $G=\R\times\R$ on canonical systems and then an orbit $\mathcal O = \{ g\cdot H_0 \}$ of this action.
Assume that there is an associated
$\textrm{\rm SL}(2,\C)$-cocycle $T(t,x;H)$. Suppose that the second copy of $\R$ acts by twisted shifts,
and the cocycle extends the one from Theorem \ref{T5.1} in the sense that
$T(0,x;H)=T_0(x;H)T_1(x;H)$. As for the first copy of $\R$, we assume that $T(t,0;H)$ satisfies the usual equation
$\dot{T}=B(t\cdot H)T$, $T(0)=1$, with $B$ bounded on $\mathcal O$. Finally, assume that $\{ m_{\pm}(z;H): H\in\mathcal O \}$
is contained in a compact subset of $\C^+$.

Then
\[
\pm m_{\pm}(z;g\cdot H) = T(g;H)(\pm m_{\pm}(z;H)) 
\]
for all $g\in G$, $H\in\mathcal O$.
\end{Theorem}
A few remarks on this perhaps:

(1) We now use the variable $x$ (rather than $s$, as before) for the twisted shift part, which I hope is still
rather suggestive notation.

(2) While the list of assumptions may seem lengthy, I would still expect them to be very easy to verify in situations of interest,
by again referring to a combination
of compactness and continuity properties; compare the comments I made earlier on Theorem \ref{T1.4}.
For this, we need a metric on canonical systems that can serve as a replacement for \eqref{djac}, and this is actually
somewhat less straightforward here than in the Jacobi case.
\cite[Section 2]{Remcont} has a discussion of these issues for Schr\"odinger operators.

(3) By focusing on a single orbit we perhaps make the assumptions easier to verify,
but really the main point here is to avoid the issues with existence and uniqueness for general initial conditions. In other words,
we bypass this by \textit{assuming }that we have an $H_0$ that we can evolve, and then we focus on this orbit exclusively.
\begin{proof}
I'll present the argument for $m=m_+$, and what I'll actually show is that $T(g;H)m(H)\in\mathcal D(x;g\cdot H)$ for all $x>0$.
This will give the claim, by Weyl theory, as discussed above. Also, I've dropped all reference to $z\in\C^+$, which was fixed at the beginning.

By the cocycle property and since we already know, from Theorem \ref{T5.1}, that $T(0,x;H)$ updates $m$ correctly,
it suffices to treat the case $g=(t,0)$.
We are given that $T_1(x;H)m(H)\in \overline{\C^+}$ for all $x>0$, and
we want to show that, similarly, $T_1(x; t\cdot H) T(t,0;H) m(H)\in \overline{\C^+}$ for $x>0$. Here $T_1$ may be replaced by $T$
since these only differ by a matrix $T_0\in\textrm{\rm SL}(2,\R)$, which acts by an automorphism of $\C^+$ and thus has no effect on
what we're trying to establish. Then the cocycle property lets us rewrite the claim as
\[
T(t,0; x\cdot H)T(0,x; H) m(H)\in\overline{\C^+} ,
\]
or, since we know that $T(0,x; H)$ updates correctly,
\begin{equation}
\label{5.16}
T(t,0; x\cdot H) m(x\cdot H) \in \overline{\C^+} .
\end{equation}
Now by assumption $m(x\cdot H)\in K$ for some compact subset $K\subseteq\C^+$.
Moreover, for small $t>0$, our assumption on $B$ guarantees that $T(t,0;x\cdot H)=1+o(1)$ as $t\to 0+$, uniformly in $x\ge 0$.
So \eqref{5.16} will indeed
hold, at least for small $0\le t \le \delta$. Notice also that $\delta$ depends only on the compact subset $K\subseteq\C^+$
and on the bound on $B$. So we have established the claim of the theorem for these $t$, but then we can repeat the whole
argument and obtain it for $\delta\le t\le 2\delta$ also, and so forth. Here it is important that we never leave the orbit $\mathcal O$,
so neither $K$ nor the bound on $B$ change and thus the same $\delta>0$ works in all steps.
\end{proof}
Let's now try to look for such actions of $G=\R\times\R$ that extend a twisted shift together with its cocycle.
At a formal level (and this is all we will talk about here), things proceed much as before; differences arise only because
we are dealing with an action of $\R$ by (twisted) shifts rather than a shift \textit{map }(that is, an action of $\Z$).

As in Section 3 and Theorem \ref{T5.2} above,
let's denote by $B=B(H)$ the matrix function that describes the part of the cocycle corresponding to the unknown
evolution we're trying to construct; compare \eqref{todacc}. For the twisted shift part, the corresponding matrix is given by $C$
from \eqref{tsccC}.

In the cocycle identity
\[
T(t,0; x\cdot H)T(0,x; H) = T(0,x; t\cdot H) T(t,0; H) ,
\]
we differentiate with respect to $t$ and then $x$ at $t=0$, $x=0$. The first step yields
\begin{equation}
\label{5.12}
B(x\cdot H) T(0,x;H) = T(0,x; H)B(H) + \partial_t T(0,x; t\cdot H))\bigr|_{t=0} ,
\end{equation}
and then from this we obtain that
\[
(\partial_x B)(H) + B(H)C(H) = C(H)B(H) + (\partial_x X)\bigr|_{x=0} ,
\]
with $X$ denoting the second term on the right-hand side of \eqref{5.12}. We now make the additional
assumption that these two derivatives could have been taken in the reverse order, and this produces
\[
\partial_t C - \partial_x B = [B,C] ;
\]
as before, the derivatives must be interpreted as
\[
(\partial_x B)(H)\equiv \frac{d}{dx}B(x\cdot H)\bigr|_{x=0} ,
\]
and similarly
for $\partial_t C$. Or, plugging $C$ from \eqref{tsccC} into this, we may rewrite the equation as
\begin{equation}
\label{canzc}
\partial_t S + zJ\partial_t H - \partial_x B = [B,S] + z[B,JH] ,
\end{equation}
and this is the basic evolution equation (``zero curvature equation'') here.

As we discussed in Section 3, this can (and probably should) be thought of as the compatibility condition
on the individual cocycles (twisted shift and the one of the unknown flow, which is described by $B$) that guarantees
that gluing these together will result in a joint cocycle.

We now have a very general framework; for every twisted shift, which will provide a matrix function $S(H)$,
the zero curvature equation \eqref{canzc} potentially gives a hierarchy of evolutions. As in the Toda case, the cocycle that
we started out with is not just a collection of arbitrarily evolving matrix functions; on the contrary, it again
has the important property that it updates the $m$ functions correctly (that is, assuming that we are able to verify
the assumptions of Theorem \ref{T5.2}, which we don't even attempt here).

Unlike in the Toda case, there doesn't seem to be an obvious way to make \eqref{canzc} spit out hierarchies of evolutions,
and in fact some negative results have been obtained by Hur and Ong \cite{HOng}.
They discuss the case $S=0$ in \eqref{schrS}, corresponding to the plain shift, so look at the simplified equation
\begin{equation}
\label{5.14}
zJ\partial_t H - \partial_x B = z[B,JH] ,
\end{equation}
and then search for $B$'s with polynomial dependence on $z$ that satisfy this compatibility condition.
For constant $B$ (of trace zero), this recovers the flows
\[
(t\cdot H)(x) = e^{-tB^t} H(x) e^{-tB} 
\]
that conjugate $H$ by an $\textrm{\rm SL}(2,\R)$ matrix and that of course could have been
defined directly. Note that these do preserve the property that $H\ge 0$. Whether or not that will be the case
in more general situations is quite unclear. The cocycle is given by $T(t,0;H)=e^{tB}$.

So everything is in perfect order in the degree zero case, but now
Hur-Ong show that \eqref{5.14} will not yield evolutions with the desired properties
beyond degree zero. This result also confirms that the added flexibility of twisted shifts may be crucial.

So a more productive choice of $S$ would be the one from \eqref{schrS}; we already know that we have the
KdV hierarchy on the subclass of canonical systems that correspond to Schr\"odinger operators,
and now our formalism will extend this evolution to general canonical systems. This works because
we have found a way of expressing $V=(1/4)\det H''$ in terms of $H$, and in such a way that this makes
sense for arbitrary (smooth) $H$, not necessarily coming from a Schr\"odinger equation. The same procedure
can now be applied to the $B$'s from the classical KdV hierarchy to extend everything to canonical systems.

I'll leave the matter at that, but clearly the general properties of the formalism need further investigation.

Let me close this paper by making a few very quick remarks on how to adapt this material
to the case of evolutions commuting with
a twisted shift \textit{map }rather than a flow; in other words, we have an action of $\Z$ rather than of $\R$,
and so we return to the situation already familiar to us from the Jacobi case.

We proceed as above. As in \eqref{5.2}, we consider maps induced by group elements $n\in G=\Z$ of the form
\begin{equation}
\label{ts1}
(n\cdot H)(x) = M(n;H)^t H(x+n) M(n;H) ,
\end{equation}
for some matrix function $M: \Z\times\mathcal C\to \textrm{SL}(2,\R)$. This will define a group action of $\Z$ if $T=M^{-1}$ has the
cocycle property, and this is the same as saying that $T$ must be of the form
\begin{equation}
\label{ts2}
T(n;H) = \begin{cases} A((n-1)\cdot H) \cdots A(H) & n\ge 0 \\
A^{-1}(n\cdot H) \cdots A^{-1}(-1\cdot H) & n<0
\end{cases}
\end{equation}
for some matrix function $A$ taking values in $\textrm{SL}(2,\R)$.

We can now finish the discussion that I started after stating \eqref{5.2}, \eqref{5.4} and Proposition \ref{P5.1}.
In the current situation, this issue now
takes the following form: Why can't I just, conversely, pick a matrix function $A$ and obtain a cocycle and a $\Z$ action
from \eqref{ts1}, \eqref{ts2}? Of course, as before, there is the obvious concern that \eqref{ts2} will only define a cocycle if I had
the action already, but to define this via \eqref{ts1}, I need the cocycle. We can make this more explicit: having chosen an $A$,
\eqref{ts1}, \eqref{ts2} will formally produce
\[
(1\cdot H)(x) = A(H)^{-1t} H(x+1) A(H)^{-1} ,
\]
but this can be part of a group action only if the map $H\mapsto 1\cdot H$ is bijective, and it's easy to cook up $A$'s for
which it fails to be injective.

In any event, if we do have a $\Z$ action constructed in this way, then we will again call it a
\textit{twisted shift. }These twisted shift maps not only mimic what we did above, but they also again arise naturally
if we want to think of the Toda hierarchy in terms of canonical systems. Namely, if a Jacobi equation is rewritten as a canonical
system, by a procedure that is similar to the one used above for Schr\"odinger equations, then the shift on $J$ again induces
a \textit{twisted }shift on the corresponding canonical system. Everything is completely analogous to the Schr\"odinger case;
in particular, the twisting is again done by the transfer matrix
$A=\left( \begin{smallmatrix} -b_1/a_1 & 1/a_1 \\ -a_1 & 0 \end{smallmatrix} \right)$
of the Jacobi matrix at $z=0$.

Returning to the general situation then, we can next establish an analog of Theorem \ref{T5.1}: twisted shift maps come
with an $\SL$-cocycle, which, moreover, will update $\pm m_{\pm}$ correctly.

Finally, suppose again that we have an action of $G=\R\times\Z$ and an associated $\SL$-cocycle $T(g;H)$,
with $\Z$ acting by a twisted shift map and $T$ extending the twisted shift cocycle from the analog of Theorem \ref{T5.1} that
I just mentioned. Let's introduce some notation here and let's call this twisted shift cocycle $S(n;H)$ (please do not confuse this
with the $S$ from above, which was also related to the twisted shift cocycle, but in a different way). So we are then assuming
that $T(0,n;H)=S(n;H)$. 
Let $B(H)$ again be the matrix that describes the part of the cocycle corresponding to the (unknown) flow.
Then exactly the same (formal, without further clarification of the precise assumptions) calculation that led to \eqref{zctoda}
produces the zero curvature equation
\[
\dot{S}(H) = S(1\cdot H) B(H) - B(H) S(H) ;
\]
here, I've used the (perhaps confusing) short-hand notation $S(H)\equiv S(1;H)$, and this also equals
$S(H)=A(H)T_1(1;H)$, where $T_1$ again denotes the transfer matrix of the plain shift, $T'_1=zJHT_1$, $T_1(0)=1$.

\end{document}